\documentclass[12pt]{amsart}
\usepackage{amsmath}
\usepackage{dsfont}
\usepackage{mathrsfs}
\usepackage{amssymb}
\usepackage{amscd}
\usepackage{graphicx}
\newtheorem{theorem}{Theorem}[section]
\newtheorem{proposition}[theorem]{Proposition}
\newtheorem{definition}[theorem]{Definition}
\newtheorem{corollary}[theorem]{Corollary}
\newtheorem{remark}[theorem]{Remark}

\usepackage{ amssymb }
\usepackage{tikz}
\usepackage{enumerate}
\usepackage{lscape}
\usepackage{mathabx}
\usepackage{color}

\begin{document}

	\title[
    ]{A Systematic Approach to Chaos in PDEs and their Semidiscretizations}
	
	\date{\today}
	
		\author[Vargas-Moreno]{\'Alvaro Vargas-Moreno}
	\address{Institut Universitari de Matemàtica Pura i Aplicada,\newline\indent Universitat Polit\`{e}cnica de Val\`{e}ncia, \newline\indent 46022, Val\`{e}ncia, Spain.}
	\email{alvarmo1@upv.es}
	

	\begin{abstract}
		We present a systematic approach to characterize Devaney chaos in the solution of homogeneous one dimensional constant-coefficients PDEs on Herzog spaces. We show that for equations where the highest-order time derivative lacks mixed spatial derivatives, there always exists a weighted space in which the solution is chaotic. This analysis is extended to include mixed partial derivatives in the highest-order temporal term. Moreover, we link the chaotic behavior in Herzog spaces to the semidiscretization of the equations by identifying isometric isomorphisms between the two contexts,  providing also a characterization for chaos on these semidiscretizations. The work concludes with two applications that provide new insights on the dynamics of the fourth-order Moore-Gibson-Thompson equation and the viscous van Wijngaarden-Eringen equation.
	\end{abstract}
	
	\thanks {The author is supported by MCIN/AEI/10.13039/501100011033/FEDER, UE, Project PID2022-139449NB-I00. }
	
	\maketitle
	
	\section{Introduction}

	The qualitative analysis of evolution equations is a central theme in the study of partial differential equations. In this context, the theory of strongly continuous semigroups (or $C_0$-semigroups) has been very useful \cite{Pa} since it allows us to study the dynamical properties of the solution in terms of the infinitesimal generator \cite{engel-nagel}. Although classical research centered on well-posedness, stability and asymptotic behavior, chaotic dynamics in PDEs have recently gained attention from both theoretical and applied points of view. \\
    
    Chaos is a phenomenon usually associated with nonlinear maps and equations. Nevertheless, it is well known that linear operators and equations can also exhibit chaotic behavior in infinite-dimensional spaces (see \cite{bayart, Alfred}). In this context, the study of chaos in strongly continuous semigroups arises as the continuous counterpart of the well-developed theory for linear operators. In \cite{Herzog}, Herzog introduced a family of spaces of analytic functions controlled by a weight parameter to analyze Devaney chaos for the heat equation. This class of spaces has proven particularly well-suited for studying chaos in one-dimensional, homogeneous, constant-coefficient equations. Since then, many results on Devaney chaos have been established for such equations, including third- and fourth-order Moore–Gibson–Thompson equations \cite{conejero_lizama_rodenas2015chaotic, lizamamurilloMGT}, the viscous Van Wijngaarden equation \cite{conejerol_lizama_murillo2016vanwjingaarne}, and the KdV and BBM equations \cite{LMV}, among others \cite{conejero_martinez-gimenez_peris_rodenas2016chaotic, Conejero_lizama_murillo2017Chaotic, conejero_peris_trujillo2010chaotic, lizamamurillo2023, LMP, zhu, yang}.

    The aim of this work is to establish a unified framework for analyzing Devaney chaos in this class of equations. While the standard approach in Herzog spaces involves studying the full solution semigroup, which incorporates both the solution and its higher-order time derivatives, our main focus here is on the dynamics of the primary solution itself. To this end, following the strategy in \cite{MPV-chaos}, we characterize Devaney chaos for the projection family of operators associated with the solution. This approach avoids the necessity of a case-by-case analysis. Furthermore, as demonstrated in the final section for the Moore–Gibson–Thompson and viscous Van Wijngaarden equations, this general framework relaxes the technical requirements for establishing chaotic dynamics, hence extending the parameter conditions under which we obtain this phenomenon.\\

We complete the analysis of these types of equations by studying the dynamical behavior of their semidiscretization. That is, we consider the differential-difference equations that result from the semidiscretization in space of the equations, in accordance with the so-called Method of Lines (see \cite{MOL}). 
In this context, we follow recent works such as \cite{lizamamurillo2024} and \cite{slavik}, where the authors studied differential-difference equations from a functional analysis point of view.  
In our work, we characterize the dynamics in terms of Devaney chaos for the solutions of semidiscrete equations. We prove that we cannot obtain Devaney chaos for the solution semigroups by using backward difference schemes. Nevertheless, that is not the case for forward difference schemes, in which we provide necessary and sufficient conditions for chaos on the projection family of operators associated with the solution. Moreover, we prove that every abstract Cauchy problem associated with the classes of PDEs considered in this work, on a Herzog space of weight $\rho>0$, is isometrically isomorphic to the forward semidiscretization of another abstract Cauchy problem of the same type, taking the spatial step $h=\frac{1}{\rho}$. This establishes an interesting connection between the dynamics of the analytical and numerical solutions of partial differential equations.\\

We divide the analysis into two main cases. In the first case, the highest-order time derivative appears without mixed spatial derivatives. 
\begin{equation}\label{caso ecuacion polinomial}
\partial_t^n u + \sum_{\substack{0 \leq i < n \\ 0 \leq j \leq m}} a_{i,j}\,\partial_t^i \partial_x^j u = 0, \quad n, m\in\mathds{N}.
\end{equation}
In the second case, mixed partial derivatives are allowed.
\begin{equation}\label{caso ecuación racional}
 \partial_t^nu+\sum_{j=1}^m a_{n,j}\,\partial_t^n \partial_x^j u+\sum_{\substack{0 \leq i < n \\ 0 \leq j \leq m}} a_{i,j}\,\partial_t^i \partial_x^j u = 0, \quad n, m\in\mathds{N}.
\end{equation}

For each case we can transform the $n$-th order in time PDE into a system of first order in time partial differential equations. Namely, for the first case we have: 
\begin{equation}\label{sistema ecuaciones caso polinomico}
\begin{cases}
    &u_1=u\\
    &\partial_t u_1=u_2\\
    &\vdots\\
    &\partial_t u_{n-1}=u_n\\
    &\partial_tu_n=-\sum_{j=0}^ma_{n-1, j}\partial_x^ju_n-...-\sum_{j=0}^ma_{0, j}\partial_x^j u_1
\end{cases},
\end{equation}
and for the second case: 
\begin{equation}\label{sistema ecuaciones caso racional}
\begin{cases}
    &u_1=u\\
    &\partial_t u_1=u_2\\
    &\vdots\\
    &\partial_t u_{n-1}=u_n\\
    &\partial_t(I+\sum_{j=1}^ma_{n,j}\partial_x^j)u_n=-\sum_{j=0}^ma_{n-1, j}\partial_x^ju_n-...-\sum_{j=0}^ma_{0, j}\partial_x^j u_1
\end{cases}.
\end{equation}

This work is organized as follows. In Section 2 we recall some preliminaries on dynamics of operators, strongly continuous semigroups and some notions on resultants and discriminants of polynomial functions. Sections 3 and 4 are devoted to characterizing sub-chaos for the solution semigroups of first-order and n-th-order in time equations, respectively. Section 5 applies these results to semidiscrete models, while Section 6 provides applications to the fourth-order Moore–Gibson–Thompson and viscous van Wijngaarden–Eringen equations.
	
\section{Preliminaries}\label{preliminaries}
		In this section, we recall some notions and results about linear dynamics, strongly continuous semigroups of operators, functional calculus and discriminants and resultants of polynomial functions.  
		
		\subsection{Dynamics of linear operators} \vspace{12pt}
		An operator $T$ on a topological vector space $X$ is called \textit{hypercyclic} if there is a vector $x\in X$ such that its orbit $\text{Orb}(x, T)=\{x, Tx, T^2 x, ...\}$ is dense in $X$. 
		Furthermore, an operator $T$ on a topological vector space $X$ is said to be \textit{Devaney chaotic} if it is hypercyclic and the set of periodic points $\text{Per}(T)$ is dense in $X$. It is a very well known result \cite{Banks-Brook} that a map $T$ which is chaotic on a metric space $X$ without isolation points is \textit{sensitive}, i.e. there exists $\delta>0$ such that for each $x\in X$ and each $\epsilon>0$, there are $y\in X$ and $n\in\mathds{N}$ with $d(x, y)<\epsilon$ and $d(T^n(x), T^n(y))>\delta$. Finally, an operator is said to be \textit{sub-chaotic} if there is an invariant closed subspace $Y\neq\{0\}$ such that $(T|_{Y}): Y\rightarrow Y$ is Devaney chaotic. \
        
        In this paper, we will work with $n$-th order in time PDEs in the form of systems of first order partial differential equations, as in \eqref{sistema discreto ecuaciones caso polinomico} and \eqref{sistema discreto ecuaciones caso racional}. Therefore, following \cite{MPV-chaos}, we will study the dynamics of the first coordinate projection of the solution semigroup. This approach is taken to focus exclusively on the dynamics of the solution of the original equation, rather than the evolution of the solution and its derivatives. This motivates the following definitions. \
        \begin{definition}
 A sequence of operators $T_n: Z\rightarrow X$, $n\in\mathds{N}_0$, is hypercyclic (or universal) if there exists $z\in Z$ such that $\text{Orb}(z, (T_n)):=\{T_nz: n\in\mathds{N}_0\}$ is dense in $X$. 
	\end{definition}
	\begin{definition}\label{projseq}
 Let $\mathcal{T}: X^n\rightarrow X^n$ be an operator, then we define the projection sequence of operators  $T_n: X^n\rightarrow X$ given by 
 $$
 T_n(x_1,..., x_n):=\pi(\mathcal{T}^n(x_1,..., x_n)), \ \ \mbox{ for } (x_1, ..., x_n)\in X^n, 
 $$
where $\pi: X\times X\rightarrow X$ is the first coordinate projection operator.

 Within this framework, $(T_n)_n$ is said to be Devaney chaotic if it satisfies the following two conditions: 
		\begin{enumerate}[(i)]
			\item There exists a $\mathcal{T}$-invariant subspace $Y$ of $X^n$ with $\overline{\pi (Y)}=X$. 
			\item $\mathcal{T}$ is sub-chaotic with respect to $Y$.
		\end{enumerate}
	\end{definition}
 
 \begin{remark}
   Observe that if a projection sequence $(T_n)_n$ as considered in Definition~\ref{projseq} is Devaney chaotic then the sequence of operators $T_n: X ^n\rightarrow X$, $n\in\mathds{N}_0$, is hypercyclic, and the set $\{ x\in X \ : \ \exists (z_1, ..., z_{n-1})\in X^{n-1} \ \mbox{ with } \ (x,z_1, ..., z_{n-1})\in\text{Per}(\mathcal{T})\}$ is dense in $X$. 
 \end{remark} 
		Given an operator $T:X\rightarrow X$ on a complex Banach space $X$, a function $E:A\rightarrow X$ for certain $A\subset \mathds{C}$ is an \textit{eigenvector field} if $E(\lambda)\in\text{ker}(\lambda I-T)$ for any $\lambda\in A$.  The following criterion adapts the classical Eigenvector Field criterion \cite{bayartgrivauxhyperci} to the context of sub-chaos of operators. 

	\begin{theorem}\label{eigenvector-criterion operador} 
 Let $T: X\rightarrow X$ be an operator on a complex Banach space X. If $U\subset \mathds{C}$ is a nonempty connected open set such that $U\cap\mathds{T}\neq \emptyset$ and $G:U\rightarrow X$ is a weakly holomorphic eigenvector field of $T$, then the restriction of $T$ to the invariant subspace
 $$Y:=\overline{\text{span}\{G(\lambda): \lambda\in U\}}$$
 is Devaney chaotic. Moreover, if $Y=X$ then $T$ is Devaney chaotic.  
	\end{theorem}
    The following is a necessary condition for (sub)-chaos. 
    \begin{theorem}\label{condicion necesaria caos operador}\cite[Prop. 5.7]{Alfred} Let $T$ be a Devaney chaotic operator on a Banach space $X$. Then $\sigma(T)$ has no isolated points and its point spectrum $\sigma_P(T)$ contains infinitely many roots of unity.      
    \end{theorem}
		
		\subsection{Dynamics of strongly continuous semigroups}
		
		We now recall the corresponding notions in linear dynamics for  strongly continuous semigroups of operators. 
		
		\begin{definition}
			Let $X$ be a Banach space. A strongly continuous semigroup ($C_0$-semigroup) is a one-parameter family $(T_t)_{t\geq 0}\subset \mathcal{B}(X)$ of operators on $X$ such that the following assertions hold: 
            \begin{enumerate}[(i)]
                \item $T_0=I.$
                \item $T_{t+s}=T_t\circ T_s$ for all $t, s\geq 0.$
                \item $\lim_{s\rightarrow t}T_s x=T_t x$ for all $x\in X$ and $t\geq 0$. 
            \end{enumerate} 
            The operator 
			$$Ax:=\lim_{t\rightarrow 0}\frac{1}{t}(T_tx-x),$$
			exists on a dense subspace of $X$ denoted by $D(A)$; the so-called domain of $A$ and  $(A,D(A))$ is called the infinitesimal generator of the semigroup.
		\end{definition}
		
		As a consequence of the Hille-Yosida theorem \cite[Theorem 7.4]{brezis2011functional},  the solution of the abstract Cauchy problem on $X$ given by: 
		\begin{equation}\label{abcauchy}
			\left\{\begin{array}{ll}
				\partial_t u(t)=A u(t)\\
				u(0)=\varphi,
			\end{array}
			\right.
		\end{equation}
		is $u(t)= T_t \varphi$ whenever $\varphi \in D(A)$. If $A\in\mathcal{B}(X)$, that is, $A$ is a bounded linear operator, then the semigroup is \textit{uniformly continuous} and can be represented as $T_t=e^{tA}=\sum_{k=0}^{\infty}(tA)^n/n!$ for all $t\geq 0$ (see \cite[Ch. I, Prop. 3.5]{engel-nagel}). We  recall the definition of Devaney chaos for a strongly continuous semigroup.
		\begin{definition}
		An element $x\in X$ is  called a periodic point for  $(T_t)_{t\geq 0}$ if there exists some $t>0$ such that $T_tx=x$. 
		A $C_0$-semigroup $(T_t)_{t\geq 0}$ is called Devaney chaotic if  there exists $x\in X$ such that the set $\{T_tx:t\geq 0\}$ is dense in $X$ and the set of periodic points is dense in $X$. It is said to be sub-chaotic if there exists a closed subspace $Y\neq \{0\}$ invariant under $(T_t)_{t\geq 0}$, such that $(T_t|_Y)_{t\geq 0}$ is Devaney chaotic as a $C_0$-semigroup on $Y$. 
	\end{definition}
    As in the discrete case, we also define the projection family of operators in the context of strongly continuous semigroups. 
    \begin{definition}\label{projfam}
 Let $\mathcal{T}_t: X^n\rightarrow X^n$, $t\geq 0$, be a strongly continuous semigroup. Then we define the projection family of operators  $T_t: X ^n\rightarrow X$, $t\geq 0$, given by 
 $$
 T_t(x_1, ..., x_n):=\pi(\mathcal{T}_t(x_1, ..., x_n)), \ \ \mbox{ for } \ (x_1, ..., x_n)\in X^n, 
 $$
where $\pi: X^n\rightarrow X$ is the first coordinate projection operator.  $(T_t)_{t\geq 0}$ is said to be Devaney chaotic if $(\mathcal{T}_t)_{t\geq 0}$ is sub-chaotic with respect to a closed subspace $Y\neq \{0\}$, invariant for the $C_0$-semigroup, with $\overline{\pi (Y)}=X$. 
	\end{definition}

		The following criterion, introduced in \cite{BM2}, provides sufficient conditions to ensure  sub-chaos for $C_0$-semigroups.  We refer the reader to the seminal paper \cite{desch_schappacher_webb1997hypercyclic}, where the original  Desch-Schappacher-Webb criterion for chaos on $C_0$-semigroups was stated.

	\begin{proposition}\label{subcaos semigrupos}\cite{BM2} Let $X$ be a complex separable infinite-dimensional Banach space and let $(A, D(A))$ be the generator of a $C_0$-semigroup $(T_t)_{t\geq 0}$ on $X$. Assume that there exists an open connected subset $U$ and a weakly holomorphic function $f:U\rightarrow X$ such that 
		\begin{enumerate}[(i)]
			\item $U\cap i\mathds{R}\neq\emptyset,$
			\item $f(\lambda)\in\text{ker}(\lambda I-A)$ for every $\lambda\in U.$
		\end{enumerate}
		Then the restriction of the $C_0$-semigroup $(T_t)_{t\geq 0}$ to the invariant space
		$$X_U:=\overline{\text{span}\{f(\lambda):\lambda\in U\}}$$
		is Devaney chaotic. In particular, the $C_0$-semigroup $(T_t)_{t\geq 0}$ is sub-chaotic. Also, if $X_U=X$ then  $(T_t)_{t\geq 0}$ is chaotic.
	\end{proposition}
    
        Let us also recall the following necessary condition for sub-chaos on strongly continuous semigroups. 
        \begin{theorem} \cite[Th. 7.18]{Alfred}\label{necessary condition chaos} Let $(A, D(A))$ be the generator of a sub-chaotic strongly continuous semigroup on $X$. Then 
        $$\sigma_P(A)\cap i\mathds{R}$$
        is infinite. 
            
        \end{theorem}
		Finally, we recall the definition of the Banach space of analytic functions of Herzog type \cite{Herzog}. Given $\rho>0$, let: 
		$$X_{\rho}=\left\{f:\mathds{R}\rightarrow\mathds{C}: f(x)=\sum_{n=0}^{\infty}\frac{a_n\rho^n}{n!}x^n, (a_n)_{n\geq 0}\in c_0(\mathds{N}_0)\right\}$$
		endowed with the norm $\|f\|=\sup_{n\geq 0}|a_n|.$ This space is isometrically isomorphic to $c_0(\mathbb{N}_0)=\{a:\mathds{N}_0\rightarrow \mathds{C}:\lim_{n\rightarrow\infty}|a_n|=0\}.$
\subsection{Functional calculus}
Let us introduce some useful results on the Riesz-Dunford functional calculus and spectral theory. The standard definitions and properties can be found in chapter VII of \cite{dunford-schwartz}. Throughout the work, we consider $\sigma(T)$ and $\sigma_P(T)$ to be the spectrum and the point spectrum of an operator.  
		\begin{definition} Let $X, Y$ be Banach spaces and $T:X\rightarrow Y$ a bounded operator. We denote $\mathcal{F}(T)$ as the set of holomorphic functions in a neighborhood of $\sigma(T)$. 
		\end{definition}
		\begin{definition} Let $f\in\mathcal{F}(T)$ and $U\supset\sigma(T)$ be an open set whose boundary $\partial U $ consists of a finite number of rectifiable curves. Suppose $U \cup \partial U$ is in the domain of analyticity of $f$. Then 
			$$f(T):=\frac{1}{2\pi i}\int_{\partial U }f(\lambda)R(\lambda; T)d\lambda,$$
			is well defined, where $R(\lambda; T):=(\lambda I-T)^{-1}$ is the resolvent operator defined in $\rho(T):= \mathds{C}\setminus\sigma(T)$, the resolvent set of $T$.
		\end{definition}
		As a consequence of the Cauchy integral theorem and  the fact that the resolvent operator is holomorphic, it follows that $f(T)$ is independent of the domain $U$. 
		\begin{theorem}\label{spectral mapping}\cite[Ch. VII.3, Th. 11]{dunford-schwartz}(Spectral Mapping Theorem). Let $f\in\mathcal{F}(T)$, then $\sigma_P(f(T))=f(\sigma_P(T))$ and $\sigma(f(T))=f(\sigma(T))$ .   
		\end{theorem}

If we consider a bounded operator $A$ on a Banach space $X$, then by Theorem \ref{spectral mapping}, the uniformly continuous semigroup $\left(e^{tA}\right)_{t\geq 0}$ is such that
  \begin{equation*}
      \sigma(e^{At})= e^{\sigma(A)t};\quad \sigma_P(e^{At})= e^{\sigma_P(A)t}, \quad t\geq 0.
  \end{equation*}
  Let us now recall some applications of the spectral mapping theorem in the context of the backward shift operator. 
		\begin{theorem}\label{th:3}
			\cite[Ch. VII Cor. 6.6]{conway-functional}\label{espectro backward shift}
			Let $X$ be one of the sequence spaces $\ell^p(\mathds{N}_0)$, $1\leq p\leq\infty$ or $c_0(\mathds{N}_0)$ and $B:X\rightarrow X$ be the backward shift operator. Then $\sigma_{p}(B)=\mathds{D}$, $\sigma_{c}(B)=\mathds{T}$ and $\sigma(B)=\overline{\mathds{D}}$. 
		\end{theorem}

\begin{theorem}\cite[Ch. 4, Th. 4.43]{Alfred}\label{criterio-backward}
    Let $\varphi$ be a nonconstant holomorphic function on a neighborhood of $\overline{\mathds{D}}$. Then the following assertions are equivalent: 
    \begin{enumerate}[(i)]
        \item $\varphi(B)$ is chaotic;
        \item $\varphi(\mathds{D})\cap\mathds{T}\neq \emptyset;$
        \item $\varphi(B)$ has a nontrivial periodic point. 
    \end{enumerate}
\end{theorem}
\subsection{Polynomial resultants and discriminants}
In this section we introduce some classical notions on polynomial functions that will be of further use in this work. Throughout this work $\mathds{K}$ will denote either $\mathds{C}$ or $\mathds{R}$.

\begin{definition} Let $p, q \in \mathds{K}[x]$ be two polynomials of degree $n$ and $m$, respectively, given by
$$p(x) = a_n x^n + a_{n-1} x^{n-1} + \dots + a_0, \quad a_n \neq 0,$$
$$q(x) = b_m x^m + b_{m-1} x^{m-1} + \dots + b_0, \quad b_m \neq 0.$$
The \textit{resultant} of $p$ and $q$, denoted by $\textnormal{res}(p, q)$, is defined as the determinant of the $(n+m) \times (n+m)$ Sylvester matrix:
$$\text{res}(p, q) := \det 
\begin{pmatrix}
a_n      & 0        & \cdots & 0        & b_m      & 0        & \cdots & 0\\
a_{n-1}  & a_n      & \cdots & 0        & b_{m-1}  & b_m      & \cdots & 0\\
\vdots   & \vdots   & \ddots & \vdots   & \vdots   & \vdots   & \ddots & \vdots\\
a_0      & a_1      & \cdots & a_n      & b_0      & b_1      & \cdots & b_m\\
0        & a_0      & \cdots & a_{n-1}  & 0        & b_0      & \cdots & b_{m-1}\\
\vdots   & \vdots   & \ddots & \vdots   & \vdots   & \vdots   & \ddots & \vdots\\
0        & 0        & \cdots & a_0      & 0        & 0        & \cdots & b_0
\end{pmatrix}.$$
Where the first $m$ columns contain the coefficients of $p$ and the remaining $n$ columns contain the coefficients of $q$. 
\end{definition}
A fundamental property of the resultant is that it characterizes the existence of common roots. Specifically, $\text{res}(p, q) = 0$ if and only if $p$ and $q$ have a common root.\\

A closely related notion is the discriminant of a polynomial. 
\begin{definition}
    Let $p=a_nx^n+\cdots +a_0\in \mathds{K}[x]$, with $a_n\neq0$ and $n\in\mathds{N}$. Then the discriminant of $p$, denoted as $\textnormal{disc}(p)$, is defined as: 
    $$\textnormal{disc}(p):=\frac{(-1)^{n(n-1)/2}}{a_n}\textnormal{res}(p, p'),$$
    where $p'(x)=na_nx^{n-1}+\cdots +a_1$ is the derivative of $p$.
\end{definition}
By the definition of the discriminant, we note that $\textnormal{disc}(p)=0$ if and only if $\textnormal{res}(p, p')=0$. This indicates that $p$ and $p'$ have a common root, and therefore $p$ has a multiple root. 

\section{First-order in time case}
In this section, we characterize the dynamical behavior in terms of chaos of the solution semigroup of first-order in time equations of the form \eqref{caso ecuacion polinomial} and \eqref{caso ecuación racional}. In order to do so, we denote the map $\partial_x: X_{\rho}\rightarrow X_{\rho}$, as the derivative operator acting on a Herzog space $X_{\rho}$ of weight $\rho>0$. We note that by the isometric isomorphism between $X_{\rho}$ and $c_{0}(\mathds{N}_0)$, the operator $\partial_x$ is isometric to the operator $\rho B$ on $c_0(\mathds{N}_0)$. As a consequence, and by Theorem \ref{th:3} we have that
$$\sigma_P(\partial_x)=\sigma_P(\rho B)=\rho\mathds{D} \quad\text{ and }\quad \sigma(\partial_x)=\sigma(\rho B)=\rho\overline{\mathds{D}},$$
where $\mathds{D}:=\{z\in\mathds{C}: |z|<1\}$.

The following result establishes a characterization of chaos for $C_0$-semigroups generated by operators of the type $f(\partial_x)$, where $f$ is a holomorphic function on a neighborhood of $\sigma(\partial_x)=\rho\overline{\mathds{D}}$.

\begin{theorem}\label{equivalencias primer orden} Let $\rho>0$ and $f$ be a nonconstant holomorphic function on a neighborhood of $\rho\overline{\mathds{D}}$ and consider the semigroup $(T_t)_{t\geq 0}$ generated by the operator $f(\mathds{\partial}_x)$. Then the following assertions are equivalent: 
\begin{enumerate}[(i)]
    \item $f(\rho\mathds{D})\cap i\mathds{R}\neq \emptyset$. 
    \item $(T_t)_{t\geq 0}$ is Devaney chaotic in $X_{\rho}$. 
    \item Every autonomous discretization of $(T_t)_{t\geq 0}$ is Devaney chaotic in $X_{\rho}$. 
    \item Some autonomous discretization of $(T_t)_{t\geq 0}$ is Devaney chaotic in $X_{\rho}$. 
\end{enumerate}
\end{theorem}
\begin{proof}
   $(ii)\rightarrow(i)$. By the spectral Theorem we have that 
   $\sigma_P(f(\partial_x))=f(\rho\mathds{D}).$
   As a consequence, by Theorem \ref{necessary condition chaos} we have that if $(T_t)_{t\geq 0}$ is chaotic then 
$\sigma_P(f(\partial_x))\cap i\mathds{R}=f(\rho\mathds{D})\cap i\mathds{R}\neq\emptyset$. \\

$(i)\rightarrow (iii)$. Since $f(\partial_x)$ is a bounded operator in $X_{\rho}$ we have that $T_{t}=e^{t f(\partial_x)}$ for all $t\geq 0$. By the isometric isomorphism between  $X_{\rho}$ and $c_0(\mathds{N}_0)$ we have that $e^{t f(\partial_x)}$ is isometrically isomorphic to the operator $e^{t f(\rho B)}$ in $c_0(\mathds{N}_0)$. Hence by the spectral mapping Theorem we have $$\sigma_P(e^{t f(\rho B)})=e^{tf(\rho\mathds{D})} \quad\text{ for all }t\geq 0.$$
Finally by condition $(i)$ it follows that $e^{tf(\rho\mathds{D})}\cap \mathds{T}\neq \emptyset$ for all $t>0$ and therefore by Theorem \ref{criterio-backward} the operator $T_t=e^{tf(\partial_x)}$ is Devaney chaotic in $X_{\rho}$ for all $t>0$. \\

It is immediate to see that $(iii)\rightarrow(iv)$ and $(iv)\rightarrow(ii)$.
\end{proof}
\begin{remark}
    Let us observe that Theorem \ref{equivalencias primer orden} can be seen as the counterpart to Theorem \ref{criterio-backward} in the context of operator semigroups on a Herzog space. Moreover, we highlight the importance of conditions $(iii)$ and $(iv)$ in this theorem. While there exist chaotic semigroups whose autonomous discretizations are not all chaotic (see \cite{bayart-bermudez}), this is not the case for semigroups generated by operators of the type $f(\partial_x)$ on a Herzog space.
\end{remark}
The following corollaries follow as a consequence of Theorem \ref{equivalencias primer orden} and are particularly interesting in the analysis of chaos for the first-order in time case of equations \eqref{caso ecuacion polinomial} and \eqref{caso ecuación racional}.
\begin{corollary}\label{caos polinomial primer orden} Let $p$ be a nonconstant polynomial. Then there is some $\rho>0$ such that $p(\partial_x)$ generates a Devaney chaotic $C_0$-semigroup on $X_{\rho}$. 
\end{corollary}
\begin{proof}
    Without loss of generality, consider $p$ to be a polynomial function of degree $n$. Then for all $\rho>0$ we have: 
    $$p(\rho\mathds{D})=\left\{a_n\rho^nz^n+\dots+a_0, z\in\mathds{D}\right\}.$$
    Now take $\lambda\in i\mathds{R}$. By the Fundamental Theorem of Algebra there exists some $z_{\lambda}\in\mathds{C}$ such that $a_nz_{\lambda}^n+...+a_1z_{\lambda}+a_0=\lambda$. Moreover, for any $\rho>|z_{\lambda}|$ there exists some $z_0\in\mathds{D}$ such that $\rho z_0=z_{\lambda}$. As a consequence, $p(z_{\lambda})=p(\rho z_0)=\lambda\in i\mathds{R}$ and the conclusion follows by Theorem \ref{equivalencias primer orden}.
\end{proof}
\begin{corollary}\label{caos primer orden racional}
    Let $q, p$ be nonconstant polynomials of order $m$ and $n$ respectively. Then for all $\rho>0$ such that 
    $$\{z\in\mathds{C}: q(z)=0\}\subset \mathds{C}\setminus \rho\overline{\mathds{D}},$$ 
    the operator $f(\partial_x):=(p/q)(\partial_x)$ generates a chaotic semigroup on $X_\rho$ if and only if there are some $z_0\in\rho\mathds{D}$ and $ \lambda_0\in i\mathds{R}$ such that $p(z_0)=\lambda_0q(z_0)$. 
\end{corollary}
\begin{proof}
    By the hypothesis, it follows that $f$ is holomorphic on a neighborhood of $\sigma(\partial_x)=\rho\overline{\mathds{D}}$  and therefore the operator $f(\partial_x)$ is well defined on $X_{\rho}$. The conclusion follows as an immediate consequence of Theorem \ref{equivalencias primer orden}.   
\end{proof}
We also highlight the following particular case. 
\begin{corollary}\label{caos primer orden racional caso particular}
    Let $p$ be a nonconstant polynomial, $k\in\mathds{N}$, $\alpha\in\mathds{C}$. Then for each $0<\rho<\frac{1}{|\alpha|^{1/k}}$ the operator $f(\partial_x):=(I-\alpha \partial_x)^{-1}p(\partial_x^k)$ generates a chaotic semigroup on $X_\rho$ if and only if there are some $z_0\in\rho\mathds{D}, \lambda_0\in i\mathds{R}$ such that $p(z_0)=\lambda_0(1-\alpha z_0^k)$. 
\end{corollary}

We can apply the previous results to the analysis of first order in time partial derivatives. Since $\partial_x$ is a bounded operator on $X_{\rho}$ for all $\rho>0$, we can write the first-order in time case of equation \eqref{caso ecuacion polinomial} as the following abstract Cauchy problem: 
\begin{equation}\label{Cauchy primer orden polinomial}
    \begin{cases}
        \partial_t u=p(\partial_x)u\\
        
        u(0)=u_0\in X_{\rho}
    \end{cases},
\end{equation}
where $p(z):=-\sum_{j=0}^ma_{j}z^j$ is a polynomial function and $p(\partial_x)$ is a bounded operator on $X_{\rho}$ for all $\rho>0$. The following result is an immediate consequence of Corollary \ref{caos polinomial primer orden}. 
\begin{corollary}
Let a PDE be first order in time and without mixed derivatives, of the form \eqref{Cauchy primer orden polinomial}. Then there exists $\rho_0>0$ such that the associated solution semigroup is Devaney chaotic on $X_{\rho}$ for every $\rho>\rho_0$.
\end{corollary}

Now for the case of equation \eqref{caso ecuación racional}, we define the polynomials $p(z):=-\sum_{j=0}^ma_{0,j}z^j$ and $q(z):=1+\sum_{j=1}^ma_{1, j} z^j$. Applying functional calculus and recalling that $\sigma(\partial_x)=\overline{\rho\mathds{D}}$, we have that for all $\rho>0$ such that 
$$\{z\in\mathds{C}: q(z)=0\}\subset \mathds{C}\setminus\overline{\rho\mathds{D}},$$
the operator $q(\partial_x)^{-1}=(1/q)(\partial_x)$ is well defined on $X_{\rho}$. Moreover, on such conditions for the weight of the space, we can write the first-order in time case of equation \eqref{caso ecuación racional} as the following abstract Cauchy problem: 
\begin{equation}\label{Cuchy primer orden racional}
    \begin{cases}
        &\partial_t u=(p/q)(\partial_x)u\\
        &u(0)=u_0\in X_{\rho}
    \end{cases}.
\end{equation}
Applying Corollaries \ref{caos primer orden racional} and \ref{caos primer orden racional caso particular} we get the following result. 
\begin{corollary}
    Let a PDE be first order in time and with mixed derivatives as in equation \eqref{caso ecuación racional}. Then for all $\rho>0$ such that 
    $$\{z\in\mathds{C}: q(z)=0\}\subset \mathds{C}\setminus \rho\overline{\mathds{D}},$$ 
    the associated solution semigroup is chaotic on $X_\rho$ if and only if there are some $z_0\in\rho\mathds{D}$ and $ \lambda_0\in i\mathds{R}$ such that $p(z_0)=\lambda_0q(z_0)$. \\

    Moreover, if $q(\partial_x):=I-\alpha \partial_x^k$ for some $\alpha\in\mathds{C}$ and $k\in\mathds{N}$. Then, for each $0<\rho<\frac{1}{|\alpha|^{1/k}}$, the associated solution semigroup is chaotic on $X_{\rho}$ if and only if there are some $z_0\in\rho\mathds{D}$, $\lambda_0\in i\mathds{R}$ such that $p(z_0)=\lambda_0(1-\alpha z_0^k)$. 
\end{corollary}
\section{$n$-th-order in time case}
In this section, we study the spectrum and dynamical properties of a block-operator matrix depending on the backward shift operator on $\ell^p(\mathds{N}_0)$ or $c_0(\mathds{N}_0)$. Next, we relate these results to a block-operator matrix depending on $\partial_x$ in a Herzog space by the isometric isomorphism between $c_0(\mathds{N}_0)$ and $X_{\rho}$. We will apply this analysis to systems \eqref{sistema ecuaciones caso polinomico} and \eqref{sistema ecuaciones caso racional}, providing general conditions to analyze sub-chaos in both cases. 

\begin{theorem}\label{espectro de la matriz}
    Let $X$ be one of the complex sequence spaces $\ell^p(\mathds{N}_0)$ $1\leq p<\infty$ or $c_0(\mathds{N}_0)$ and $f_1, f_2, ..., f_n$ be holomorphic functions on $\overline{\mathds{D}}$ such that at least one of them is nonconstant. Define the block operator matrix $A:X^n\rightarrow X^n$ as:
        $$A:=\begin{pmatrix}
0 & I & 0 & \dots \\
0 & 0 & I & \dots \\
\vdots & \vdots & \vdots & \ddots \\
f_1(B) & f_2(B) & \dots & f_n(B) 
\end{pmatrix}.$$
Moreover, for each $\lambda\in\mathds{C}$ define the operator $F_{\lambda}:X\rightarrow X$ as: 
    $$F_{\lambda}:=f_1(B)+\lambda f_2(B)+\dots+\lambda^{n-1}f_n(B)-\lambda^nI.$$
    Then the following assertions hold: 
    \begin{enumerate}[(i)]
        \item  \begin{align*}
            &\sigma_P(A)=\{\lambda\in\mathds{C}\text{ such that }0\in\sigma_P(F_{\lambda})\}\\
            &=\{\lambda\in\mathds{C}\text{ such that exists }z\in\mathds{D}\text{ with }  f_1(z)+\lambda f_2(z)+\cdots +\lambda^{n-1}f_n(z)-\lambda^n=0\},
        \end{align*}
        \item \begin{align*}
            &\sigma(A)=\{\lambda\in\mathds{C}\text{ such that }0\in\sigma(F_{\lambda})\}\\
            &=\{\lambda\in\mathds{C}\text{ such that exists }z\in\overline{\mathds{D}} \text{ with } f_1(z)+\lambda f_2(z)+\cdots +\lambda^{n-1}f_n(z)-\lambda^n=0\}.
        \end{align*} 
    \end{enumerate}
\end{theorem}
\begin{proof}
Let us first compute the eigenvalues of the block operator matrix $A$. $$\begin{pmatrix}
0 & I & 0 & \dots \\
0 & 0 & I & \dots \\
\vdots & \vdots & \vdots & \ddots \\
f_1(B) & f_2(B) & \dots & f_n(B) 
\end{pmatrix}\begin{pmatrix}
    u_1\\
    u_2\\
    \vdots\\
    u_n
\end{pmatrix}=\lambda\begin{pmatrix}
    u_1\\
    u_2\\
    \vdots\\
    u_n
\end{pmatrix},$$
which yields the system: 
$$
\begin{cases}
u_2=\lambda u_1 \\
\vdots\\
f_1(B)u_1+f_2(B)u_1+\dots+f_n(B)u_n=\lambda u_n
\end{cases}.
$$
As a consequence we have: 
$$f_1(B)u_1+f_2(B)\lambda u_1+\dots+f_n(B)\lambda^{n-1}u_1=\lambda^n u_1.$$ 
Hence $\lambda$ is an eigenvalue of $A$ if and only if $0\in \sigma_P(F_\lambda)$, proving assertion $(i)$.

    We now prove assertion $(ii)$. Assume, on the contrary, that $\lambda\in\rho(A)$. That is, we can define the resolvent operator $(A-\lambda I)^{-1}$. By assertion $(i)$ we have that $0\notin \sigma_P(F_{\lambda})$, and therefore $F_{\lambda}$ is injective. Now, for each $x\in X$, there is some vector $(u_1, u_2, ..., u_n)\in X^n$ such that
    $$\begin{pmatrix}
-\lambda I & I & 0 & \dots \\
0 & -\lambda I & I & \dots \\
\vdots & \vdots & \vdots & \ddots \\
f_1(B) & f_2(B) & \dots & f_n(B)-\lambda I 
\end{pmatrix}\begin{pmatrix}
    u_1\\
    u_2\\
    \vdots\\
    u_n
\end{pmatrix}=\begin{pmatrix}
    0\\
    0\\
    \vdots\\
    x
\end{pmatrix}.$$
As a consequence, we have the following system: 
$$
\begin{cases}
u_2=\lambda u_1 \\
u_3=\lambda u_2=\lambda^2u_1\\
\vdots\\
u_n=\lambda u_{n-1}=\lambda^{n-1}u_1\\
f_1(B)u_1+f_2(B)u_2+\dots+(f_n(B)-\lambda)u_n=x
\end{cases}.
$$
And from the last equation it follows that: 
\begin{align*}
    &f_1(B)u_1+f_2(B)u_2+\dots+(f_n(B)-\lambda)u_n\\
    &=f_1(B)u_1+f_2(B)\lambda u_1+\dots+(f_n(B)-\lambda)\lambda^{n-1}u_1\\
    &=F_{\lambda}u_1=x.
\end{align*}
Hence we have that the operator $F_{\lambda}$ is a bijection and therefore $0\in \rho(F_{\lambda})$. \\

 On the other hand, assume that $\lambda\in\mathds{C}$ is such that $0\in\rho(F_{\lambda})$. By assertion $(i)$ it follows that $\lambda\notin \sigma_{P}(A)$, and therefore $A-\lambda I$ is injective. For each $x_n\in X$ there exists some $u_1\in X$ such that $F_{\lambda}u_1=x_n$. Hence
$$\begin{pmatrix}
-\lambda I & I & 0 & \dots \\
0 & -\lambda I & I & \dots \\
\vdots & \vdots & \vdots & \ddots \\
f_1(B) & f_2(B) & \dots & f_n(B)-\lambda I 
\end{pmatrix}\begin{pmatrix}
    u_1\\
    \lambda u_1\\
    \vdots\\
    \lambda^{n-1}u_1
\end{pmatrix}=\begin{pmatrix}
    0\\
    0\\
    \vdots\\
    x_n
\end{pmatrix}.$$
Furthermore, for each $x_i\in X$ and $i=1, 2, ..., n-1$, consider the vector $x^i\in X^n$ such that 
$$x^i_j=\begin{cases}
    x_i\quad \text{ if }j=i\\
    0\quad \text{ otherwise}
\end{cases}.$$
Then there exists some $u_1\in X$ such that 
$$F_{\lambda}u_1=-f_{i+1}(B)x_i-f_{i+2}(B)\lambda x_i-\dots-f_{n-1}(B)\lambda^{n-i-2}x_i-(f_n(B)-\lambda)\lambda^{n-i-1}x_i.$$
As a consequence, we have that: 
$$\begin{pmatrix}
-\lambda I & I & 0 &\dots & &0 &\dots  &0 \\
0 & -\lambda I & I & \dots & &0 &\dots &0\\
\vdots & \vdots & \ddots & \ddots & &\vdots& &\vdots\\
0 &0 &\dots &-\lambda I & I &0 &\dots &0\\
0 &0 &\dots &\dots & -\lambda I &I &\dots &0\\
\vdots & \vdots & \vdots &  & &\ddots &\ddots&\vdots\\
f_1(B) & f_2(B) & \dots &f_i(B) &f_{i+1}(B)&\dots &\dots& f_n(B)-\lambda I 
\end{pmatrix}\begin{pmatrix}
    u_1\\
    \lambda u_1\\
    \vdots\\
    \lambda^{i-1}u_1\\
    x_i+\lambda^{i}u_1\\
    \vdots\\
    \lambda^{n-i-1}x_i+\lambda^{n-1}u_1
\end{pmatrix}=\begin{pmatrix}
    0\\
    0\\
    \vdots\\
    x_i\\
    0\\
    \vdots\\
    0
\end{pmatrix}.$$
It follows that $A-\lambda I$ is a bijection, and therefore $\lambda\in \rho(A)$. 
\end{proof}
As a consequence of the previous result, we characterize the spectrum of the block operator matrix $A_d$ depending on $\partial_x$ in a Herzog space $X_{\rho}^n$. 
\begin{corollary}\label{espectro matriz con derivadas parciales} 
    Let $\rho>0$ and $f_1, f_2, ..., f_n$ be holomorphic functions on $\overline{\mathds{\rho D}}$ such that at least one of them is nonconstant. Define the block operator matrix $A_d:X_\rho^n\rightarrow X_\rho^n$ as:
        $$A_d=\begin{pmatrix}
0 & I & 0 & \dots \\
0 & 0 & I & \dots \\
\vdots & \vdots & \vdots & \ddots \\
f_1(\partial_x) & f_2(\partial_x) & \dots & f_n(\partial_x) 
\end{pmatrix}.$$
Moreover, for each $\lambda\in\mathds{C}$ define the operator $F^d_{\lambda}:X_\rho\rightarrow X_\rho$ as: 
    $$F^d_{\lambda}:=f_1(\partial_x)+\lambda f_2(\partial_x)+\dots+\lambda^{n-1}f_n(\partial_x)-\lambda^nI.$$
    Then the following assertions hold: 
    \begin{enumerate}[(i)]
        \item
        \begin{align*}
            &\sigma_P(A_d)=\{\lambda\in\mathds{C}\text{ such that }0\in\sigma_P(F^d_{\lambda})\}\\
            &=\{\lambda\in\mathds{C}\text{ such that exists }z\in\mathds{\rho D}\text{ with }  f_1(z)+\lambda f_2(z)+\cdots +\lambda^{n-1}f_n(z)-\lambda^n=0\},
        \end{align*}
        \item \begin{align*}
            &\sigma(A_d)=\sigma(A_d)=\{\lambda\in\mathds{C}\text{ such that }0\in\sigma(F^d_{\lambda})\}\\
            &=\{\lambda\in\mathds{C}\text{ such that exists }z\in\overline{\mathds{\rho D}} \text{ with } f_1(z)+\lambda f_2(z)+\cdots +\lambda^{n-1}f_n(z)-\lambda^n=0\}.
        \end{align*} 
        \end{enumerate}
\end{corollary}
\begin{proof}
This corollary is an immediate consequence of Theorem \ref{espectro de la matriz} and the isometric isomorphism between $\partial_x$ on $X_\rho$ and $\rho B$ on $c_0$.     
\end{proof}
For a set $f_1, f_2, ..., f_n$ of holomorphic functions on $\overline{\mathds{D}}$, define the function 
$G:\mathds{C}\times\overline{\mathds{D}}\rightarrow\mathds{C}$ as
$$G(\lambda, z):=f_1(z)+\lambda f_2(z)+\dots+\lambda^{n-1}f_n(z)-\lambda^n.$$
Based on the results of Theorem \ref{espectro de la matriz}, $G$ will be referred to as the characteristic polynomial of the block-operator matrix $A$. Observe that for each fixed $z\in\overline{\mathds{D}}$, $G(\cdot, z)$ and $\frac{\partial G}{\partial z}(\cdot, z)$ are polynomial functions in $\lambda$. As a consequence, for each $z\in\overline{\mathds{D}}$, we can compute the resultant $\textnormal{res}_{\lambda}(G, \frac{\partial G}{\partial z})(z)$. By the definition of the resultant of two polynomials as the determinant of the Sylvester matrix, and the fact that $f_i$ is holomorphic on $\overline{\mathds{D}}$ for all $i=1, 2, ..., n$, it follows that $\textnormal{res}_{\lambda}(G, \frac{\partial G}{\partial z}): \overline{\mathds{D}}\rightarrow\mathds{C}$ is holomorphic. 

\begin{theorem}\label{caos matriz A}
    Let $X$ be one of the complex sequence spaces $\ell^p(\mathds{N}_0)$ $1\leq p<\infty$ or $c_0(\mathds{N}_0)$. Moreover, let $f_1, f_2, ..., f_n$ be holomorphic functions on $\overline{\mathds{D}}$ such that there exists some $\Tilde{z}\in\mathds{D}$ with $\textnormal{res}_{\lambda}{\left(G, \frac{\partial G}{\partial_z}\right)}(\Tilde{z})\neq 0$. 
    Then the following assertions are equivalent: 
    \begin{enumerate}[(i)]
        \item The block operator matrix:
        $$A=\begin{pmatrix}
0 & I & 0 & \dots \\
0 & 0 & I & \dots \\
\vdots & \vdots & \vdots & \ddots \\
f_1(B) & f_2(B) & \dots & f_n(B) 
\end{pmatrix}$$
is Devaney sub-chaotic on $X^n$. \\
\item There are some $\lambda_0\in\mathds{T}$, $z_0\in\mathds{D}$ such that $G(\lambda_0, z_0)=0$ and $\textnormal{res}_{\lambda}\left(G, \frac{\partial G}{\partial_z}\right)(z_0)\neq 0$.
    \end{enumerate}
\end{theorem}
\begin{proof}
    $(ii)\rightarrow(i)$. We will use the Eigenvector Field Criterion \ref{eigenvector-criterion operador}.
    By Theorem \ref{espectro de la matriz} we have that 
$\lambda$ is an eigenvalue of $A$ if and only if there is some $z\in\mathds{D}$ such that the characteristic polynomial $G(\lambda, z)=0$. By the hypothesis, it follows that $\lambda_0$ is an eigenvalue of $A$. \\

Since $\textnormal{res}_{\lambda}\left(G, \frac{\partial G}{\partial z}\right)(z_0)\neq 0$, it follows that the polynomials $G(\cdot, z_0)$ and $\frac{\partial G}{\partial z}(\cdot, z_0)$ share no common root. In particular, since $G(\lambda_0, z_0)=0$ then $\frac{\partial G}{\partial z}(\lambda_0, z_0)\neq 0$. 
We notice that $G(\lambda, z)$ is a holomorphic function on $\mathds{C}\times\mathds{D}$, and therefore we can apply the Implicit Function Theorem for holomorphic functions (see Theorem 8.6 in Chapter 0 of \cite{kaup}). Consequently, there exist some neighborhood $U$ of $\lambda_0$ and a holomorphic function $g:U\rightarrow \mathds{D}$ such that $G(\lambda, g(\lambda))=0$ for all $\lambda\in U$. Moreover, $g$ is a nonconstant function on $U$ because $G(\cdot, z_0)$ is a non-zero polynomial in $\lambda$ of degree $n > 1$, so it cannot vanish identically on $U$.
 
Define the sequence $e_{\lambda}$ as
$$e_{\lambda}:=(1, g(\lambda), g(\lambda)^2, ...),$$
and the vector $v_{\lambda}\in X^n$ as
$$v_\lambda:=(e_{\lambda}, \lambda e_\lambda, ...,\lambda^{n-1}e_{\lambda}).$$
Now we prove that $v_{\lambda}$ is an eigenvector of $A$ associated with the eigenvalue $\lambda$ for each $\lambda\in U$. 
Indeed we have: 
\begin{align*}
    &A v_{\lambda}=(\lambda e_{\lambda}, ..., \lambda^{n-1}e_{\lambda}, f_1(B)e_\lambda+...+f_n(B)\lambda^{n-1}e_{\lambda})\\
    &=(\lambda e_{\lambda}, ..., \lambda^{n-1}e_{\lambda}, f_1(g(\lambda))e_{\lambda}+...+\lambda^{n-1}f_n(g(\lambda))e_{\lambda})\\
    &=(\lambda e_{\lambda}, ..., \lambda^{n-1}e_{\lambda}, G(\lambda, g(\lambda))e_{\lambda}+\lambda^{n}e_{\lambda})\\
    &=(\lambda e_{\lambda}, ..., \lambda^{n-1}e_{\lambda}, \lambda^{n}e_{\lambda})=\lambda v_{\lambda}.\\
\end{align*}

Define the function $V:U\rightarrow X^n$ as: 
$$V(\lambda):=v_{\lambda}.$$
Taking $x^*\in (X^n)^*$, by the Riesz representation Theorem there are some sequences $(a^i_j)_j$, $i=0,1, 2, ..., n$ in $\ell^q(\mathds{N}_0)$, $1\leq q\leq \infty$ such that: 
$$x^*(V(\lambda))=x^*(v_{\lambda})=\sum_{j=0}^{\infty} \overline{a_j^1}g(\lambda)^j+\lambda\sum_{j=0}^{\infty} \overline{a_j^2}g(\lambda)^j+...+\lambda^{n-1}\sum_{j=0}^{\infty} \overline{a_j^n}g(\lambda)^j.$$
It follows that $V$ is a weakly holomorphic function that verifies the hypothesis of the Eigenvector Field Criterion, and the conclusion holds.\\

$(i)\rightarrow(ii)$. Recall that since $A$ is sub-chaotic $\sigma_P(A)$ contains infinitely many roots of unity. Hence by Theorem \ref{espectro de la matriz}, the set 
$$\{\lambda\in \mathds{T} \text{ such that exists }z\in\overline{\mathds{D}} \text{ with }G(\lambda, z)=0\}$$
is infinite. 

\textit{Claim 1:} The set 
$$F_1:=\{\lambda\in\mathds{C}\text{ such that }G(\lambda, \cdot)=0 \}$$
contains at most $n$ elements. Indeed, by the hypothesis on the resultant there is at least one  $\tilde{z}\in\overline{\mathds{D}}$ such that $G(\cdot, \tilde{z})$ is not identically zero. By the Fundamental Theorem of Algebra there are at most $n$ roots $\lambda_1, ..., \lambda_n$ of the polynomial $G(\cdot, \tilde{z})$. As a consequence, if $\lambda\in F_1$ then in particular $G(\lambda, \tilde{z})=0$, and therefore $\lambda=\lambda_i$ for some $i=1, 2, ..,n$.\\

By the first claim we now have that there is some $\lambda_1\in\mathds{T}$ and $z_1\in\overline{\mathds{D}}$ such that $G(\lambda_1, z_1)=0$ and $G(\lambda_1, \cdot)$ is not identically zero.

\textit{Claim 2: } there are some open neighborhoods $U$ of $\lambda_1$ and $V$ of $z_1$ such that $\overline{V}\subset \mathds{D}$ and for all $\lambda\in U$ there is some $z\in V$ with $G(\lambda, z)=0$. 

 By the Identity Theorem for holomorphic functions, $z_1$ is an isolated point of $G(\lambda_1, \cdot)$, and therefore there exists some $r>0$ such that $\overline{B(z_1, r)}\subset \mathds{D}$ and for all $z\in \overline{B(z_1, r)}$ with $z\neq z_1$, $G(\lambda_1, z)\neq 0$. Define $\partial B(z_1, r):=\{z\in\mathds{C}\text{ such that }|z-z_1|=r\}$. Since $\partial B(z_1, r)$ is a compact set and $G(\lambda_1, z)\neq 0$ for all $z\in\partial B(z_1, r)$, there exists some $m>0$ such that 
 $$m:=\min_{z\in\partial B(z_1, r)}|G(\lambda_1, z)|.$$
 
 Now since $G(\lambda, z)$ is continuous, for each $z\in\partial B(z_1, r)$ there is some $\delta_z>0$ such that if $|\lambda-\lambda_1|<\delta_z$ and $|\zeta-z|<\delta_z$ then $$|G(\lambda, \zeta)-G(\lambda_1, z)|<m/2.$$
 As a consequence of the compactness of $\partial B(z_1, r)$, there are some some $\zeta_1, \zeta_2, ..., \zeta_k$ such that
 $$\partial B(z_1, r)\subset\bigcup_{i=1}^kB(\zeta_i, \delta_{\zeta_i}).$$
 Take $\delta:=\min\{\delta_{\zeta_i}: i=1, 2, ..., k\}$ and define $U:=B(\lambda_1, \delta)$. We then have that for all $z\in\partial B(z_1, r)$, $z\in B(\zeta_i, \delta_{\zeta_i})$ for some $i=1, 2, ..., k$. Hence for all $\lambda \in U$: 
 $$|G(\lambda, z)-G(\lambda_1, z)|\leq |G(\lambda, z)-G(\lambda_1, \zeta_i)|+|G(\lambda_1, \zeta_i)-G(\lambda_1, z)|<m. $$
 Finally, we have that for each $\lambda\in U$, $G(\lambda, \cdot)$ and $G(\lambda_1, \cdot)$ are holomorphic functions such that 
 $$|G(\lambda, z)-G(\lambda_1, z)|<m\leq |G(\lambda_1, z)| \text{ for all }z\in\partial B(z_1, r),$$
 and by Rouche's Theorem $G(\lambda, \cdot)$ and $G(\lambda_1, \cdot)$ have the same number of zeros on $B(z_1, r)$. Hence for each $\lambda\in U$ there is some $z\in B(z_1, r)\subset\mathds{D}$ such that $G(\lambda, z)=0$, and the claim holds. \\

 \textit{Claim 3: }there is some $(\lambda_0, z_0)\in \mathds{T}\times \overline{B(z_1, r)}$ such that $G(\lambda_0, z_0)= 0$ and $\frac{\partial G}{\partial z}(\lambda_0, z_0)\neq 0$. Indeed, by the second claim for all $\lambda\in \mathds{T}\cap U$ there is some $z$ in $B(z_1, r)$ such that $G(\lambda, z)=0$. Take an infinite sequence $(\lambda_i, z_i)_{i\in\mathds{N}}$ such that $G(\lambda_i, z_i)=0$ for all $i\in\mathds{N}$. Since $\mathds{T}\cap U$ and $B(z_1, r)$ are relatively compact sets we have that there is a subsequence $(\lambda_{i_k}, z_{i_k})_{k\in\mathds{N}}$ such that
 $$\lim_{k\rightarrow\infty} (\lambda_i, z_i)=(\tilde{\lambda}_0, \tilde{z}_0)\in \mathds{T}\cap\overline{U}\times \overline{B(z_1, r)},$$
 and by continuity of $G$, $G(\tilde{\lambda}_0, \tilde{z}_0)=0$. If we consider $\text{res}_{\lambda}(G, G_z)(\tilde{z}_0)\neq 0$, then $\frac{\partial G}{\partial z}(\tilde{\lambda}_0, \tilde{z}_0)\neq 0$ and the conclusion follows immediately. Otherwise, since by the hypothesis $\text{res}_{\lambda}(G, G_z)$ is not identically zero in $\overline{\mathds{D}}$, $\tilde{z}_0$ is an isolated root of $\text{res}_{\lambda}(G, G_z)$, and therefore there is some $K>0$ such that $$\text{res}_{\lambda}{(G, G_z)}(z_{i_k})\neq 0\text{ for all }k>K.$$
 Hence $G(\tilde{\lambda}_0, z_{i_k})=0$ while $\frac{\partial G}{\partial z }(\tilde{\lambda}_0, z_{i_k})\neq 0$ for all $k>K$, thus yielding the claim. 

\end{proof}
\begin{corollary} The project sequence $(T_n: X^n\rightarrow X)_{n\in\mathds{N}}$ generated by the block operator matrix $A$ is Devaney chaotic if and only if condition $(ii)$ on Theorem \ref{caos matriz A} holds.  
\end{corollary}
\begin{proof}
By the definition, if the project sequence of operators $(T_n)_{n\in\mathds{N}}$ is chaotic then the operator $A$ is sub-chaotic and therefore by Theorem \ref{caos matriz A} condition $(ii)$ holds. \\

On the other hand, if condition $(ii)$ holds, then by the proof of Theorem \eqref{caos matriz A} there are some open neighborhood $U$ of $\lambda_0$ and a nonconstant holomorphic function $g:U\rightarrow \mathds{D}$ such that the function $V:U\rightarrow X^n$ with $V(\lambda)=v_{\lambda}$ is a weakly holomorphic eigenvector field. Defining $Y:=\overline{\text{span}\{V(\lambda): \lambda\in U\}}$ we have that $Y$ is a closed subset of $X^n$, invariant under $A$, such that $A|_{Y}$ is chaotic. Moreover, taking the first coordinate projection of $Y$: 
$$\pi(Y)=\pi(\overline{\text{span}\{V(\lambda): \lambda\in U\}})=\overline{\text{span}\{e_{\lambda}\}}.$$
Now, we prove that $\overline{\text{span}\{e_{\lambda}\}}=X$. Indeed, assume that there is some $x^*\in X^*$ such that $x^*(V(e_{\lambda}))=0$ for all $\lambda\in U$. Then by the Riesz representation Theorem there is some sequence $(a_n)_n\in \ell^q(\mathds{N}_0)$, $1\leq q\leq \infty$ such that: 
$$x^*(e_{\lambda})=\sum_{j=0}^{\infty}\overline{a_j}g(\lambda)^j=0\quad \text{ for all }\lambda \in U.$$
Finally, by the Open Mapping Theorem for holomorphic functions, we have that $\{g(\lambda): \lambda\in U\}$ is an open set and therefore by the Identity Theorem for holomorphic functions it follows that $a_n=0$ for all $n\in\mathds{N}_0$. The conclusion then follows by the Hahn-Banach Theorem. 

\end{proof}

For a set $f_1, f_2, ..., f_n$ of holomorphic functions on $\overline{\rho\mathds{D}}$, define the characteristic polynomial 
$G_d:\mathds{C}\times\mathds{\rho\overline{D}}\rightarrow\mathds{C}$ as
$$G_d(\lambda, z):=f_1(z)+\lambda f_2(z)+\dots+\lambda^{n-1}f_n(z)-\lambda^n.$$
The following Theorem characterizes sub-chaos for the semigroup generated by the block-operator matrix $A_d$. 
\begin{theorem}\label{caos matriz semigrupos}
    Let $\rho>0$ and $f_1, f_2, ..., f_n$ be holomorphic functions on $\overline{\rho\mathds{D}}$ such that there exists some $\Tilde{z}\in\rho\mathds{D}$ with $\textnormal{res}_{\lambda}{\left(G_d, \frac{\partial G_d}{\partial_z}\right)}(\Tilde{z})\neq 0$. 
    Then the following assertions are equivalent: 
    \begin{enumerate}[(i)]
        \item The block operator matrix:
        $$A_d=\begin{pmatrix}
0 & I & 0 & \dots \\
0 & 0 & I & \dots \\
\vdots & \vdots & \vdots & \ddots \\
f_1(\partial_x) & f_2(\partial_x) & \dots & f_n(\partial_x) 
\end{pmatrix}$$
generates a sub-chaotic semigroup $(e^{tA})_{t\geq 0}$ on $X_\rho^n$ such that the projection family $(\pi(e^{tA_d}): X_{\rho}^n\rightarrow X_{\rho})_{t\geq 0}$ is chaotic.  \\
\item There are some $\lambda_0\in i\mathds{R}$, $z_0\in\rho\mathds{D}$ such that $G_d(\lambda_0, z_0)=0$ and $\textnormal{res}_{\lambda}{\left(G_d, \frac{\partial G_d}{\partial_z}\right)}(z_0)\neq 0$.\\
\item Every autonomous discretization of $(e^{tA_d})_{t\geq 0}$ is Devaney sub-chaotic on $X_{\rho}^n$. \\
\item Some autonomous discretization of $(e^{tA_d})_{t\geq 0}$ is Devaney sub-chaotic on $X_{\rho}^n$. 
    \end{enumerate}
\end{theorem}
\begin{proof}
    $(i)\leftrightarrow (ii)$. The proof is analogous to the one presented in Theorem \ref{caos matriz A} by noting that $\sigma_P(\partial_x)=\rho\mathds{D}$ and $\sigma(\partial_x)=\overline{\rho\mathds{D}}$ on $X_\rho^n$, and using the Desch-Schappacher-Webb criterion \ref{subcaos semigrupos} for Devaney sub-chaos of strongly continuous semigroups. \\

    $(ii)\rightarrow (iii).$ We will use the Eigenvector Field Criterion \ref{eigenvector-criterion operador}. Take $t>0$. By the hypothesis, proceeding as in the proof of Theorem \ref{caos matriz A}, there is an open neighborhood $U$ of $\lambda_0$ and $g: U_0\rightarrow \rho\mathds{D}$ holomorphic such that $G_d(\lambda, g(\lambda))=0$ for all $\lambda\in U_0$. 

    For each $\lambda\in U_0$, define the vector $v_{\lambda}$ as: 
    $$v_\lambda:=\left(e^{g(\lambda)}, \lambda e^{g(\lambda)}, ..., \lambda^{n-1}e^{g(\lambda)}\right),$$
    and observe that since $g(\lambda)\in\rho\mathds{D}$ for all $\lambda\in U$, it follows that $v_\lambda\in X_{\rho}^n$. Then we can compute:  
    \begin{align*}
    &A_d v_\lambda=\left(\lambda e^{g(\lambda)}, .., \lambda^{n-1}e^{g(\lambda)}, f_1(\partial_x)e^{g(\lambda)}+...+\lambda^{n-1}f_n(\partial_x)e^{g(\lambda)}\right)\\
    &=\left(\lambda e^{g(\lambda)}, .., \lambda^{n-1}e^{g(\lambda)}, f_1(g(\lambda))e^{g(\lambda)}+...+\lambda^{n-1}f_n(g(\lambda))e^{g(\lambda)}\right)\\
    &=\left(\lambda e^{g(\lambda)}, .., \lambda^{n-1}e^{g(\lambda)}, G_d(\lambda, g(\lambda))e^{g(\lambda)}+\lambda^n e^{g(\lambda)}\right)\\
    &=\left(\lambda e^{g(\lambda)}, .., \lambda^{n-1}e^{g(\lambda)}, \lambda^n e^{g(\lambda)}\right)=\lambda\left(e^{g(\lambda)}, .., \lambda^{n-2}e^{g(\lambda)}, \lambda^{n-1} e^{g(\lambda)}\right)=\lambda v_{\lambda}.
\end{align*}
And therefore, $v_{\lambda}$ is an eigenvector of $A_d$ associated with the eigenvalue $\lambda$.
As a consequence: 
$$e^{tA_d}v_{\lambda}=\sum_{n=0}^\infty \frac{t^n}{n!}A_d^n v_s=\sum_{n=0}^\infty \frac{t^n}{n!}(\lambda^n)v_{\lambda}=e^{t\lambda}v_{\lambda},$$
so that $v_{\lambda}$ is an eigenvector of $e^{tA_d}$ associated to the eigenvalue $e^{t\lambda}$. 
    
    Now assume without loss of generality that there is some $k\in\mathds{N}$ such that $$(k-1)\pi<t|\text{Im}(\lambda)|<(k+1)\pi \text{ for all } \lambda\in U_0.$$
   It then follows that the exponential function $e^{t\lambda}$ is univalent in $U_0$. Define the set $U:=e^{tU_0}$, which is an open set such that $U\cap\mathds{T}\neq 0$. Moreover, define the horizontal strip $S_k:=\{z\in\mathds{C}: (k-1)\pi<|\text{Im}(z)|<(k+1)\pi\}$, so that we can take the branch of the logarithm $\log: \mathds{C}\setminus{\mathds{R}_{-}}\rightarrow S_{k}$:
    $$\log(z):=\log(|z|)+i(\text{Arg}(z)+k\pi),$$
    noting that $\frac{\log}{t}:U\rightarrow U_0$ is a holomorphic function such that $\frac{ \log(e^{t\lambda})}{t}=\lambda$ for all $\lambda\in U_0$.
    
    Finally, we have that for each $s\in U$, the vector $v_{s}$: 
    $$v_s:=\left(e^{g(\log(s)/t)}, \frac{\log(s)}{t} e^{g(\log(s)/t)}, ..., \frac{\log(s)^{n-1}}{t^{n-1}}e^{g(\log(s)/t)}\right),$$
    is an eigenvector of $A_{d}$ associated to the eigenvalue $s$. And therefore, we can define the function $V:U\rightarrow X_{\rho}^n$ as $V(s):=v_s.$ By the isometric isomorphism between $X_{\rho}$ and $c_0$ it is easy to check that $V$ is a weakly holomorphic function that verifies the hypothesis of the Eigenvector Field Criterion, and since $t>0$ was arbitrary the conclusion follows.\\

$(iii)\rightarrow (iv)$ and $(iv)\rightarrow (i)$ are immediate.
\end{proof}
The following corollaries study sub-chaos for the semigroup generated by the operator $A_d$ in the particular cases in which the functions $f_i$ are either polynomial or rational functions. The characterizations are based on the discriminant of polynomials of two complex variables. Therefore, if we consider $P(\lambda,z)$ to be a polynomial of two variables, we establish the notation $\text{disc}_z(P)(\lambda)$ for the discriminant of the polynomial $P(\lambda, \cdot)$ for each fixed $\lambda\in\mathds{C}$. Thus, by the definition of the discriminant, $\text{disc}_z(P)(\lambda)$ is in turn a polynomial in the variable $\lambda$.
\begin{corollary}\label{corolario para polinomios} Let $p_1, p_2, \dots, p_n$ be polynomials such that at least one of them is nonconstant. If $\textnormal{disc}_z(G_d)(\lambda)\nequiv 0$, then there is some $\rho>0$ such that the semigroup $(T_t)_{t\geq 0}$ generated by the block matrix operator 
    $$A_d:= \begin{pmatrix}
0 & I & 0 & \dots \\
0 & 0 & I & \dots \\
\vdots & \vdots & \vdots & \ddots \\
p_1(\partial_x) & p_2(\partial_x) & \dots & p_n(\partial_x) 
\end{pmatrix}$$
is Devaney sub-chaotic on $(X_p)^n$. Moreover, the projection family $(T_t: X_{\rho}^n\rightarrow X_{\rho})_{t\geq 0}$ generated by $A$ is chaotic. 
\end{corollary}
\begin{proof}
 Let us assume without loss of generality that $p_1, p_2, \dots, p_n$ are polynomials of degree at most $m$. Then we can write: 
$$G_d(\lambda, z)=-\lambda^n+\lambda^{n-1}p_n(z)+\dots+\lambda p_2(z)+p_1(z)=a_m(\lambda)z^m+\dots+a_1(\lambda)z+a_0(\lambda),$$
where $a_0, a_1, \dots, a_m$ are polynomials of degree at most $n$. Note that $a_0(\lambda) = -\lambda^n + \dots$ is not identically zero, so $G_d(\lambda, z)$ is not identically zero as a polynomial in $z$ for any $\lambda \in \mathds{C}$. 
Consider the polynomial function $\text{disc}_z(G_d)(\lambda)$. By hypothesis and the Fundamental Theorem of Algebra, there are at most some $\lambda_1, \dots, \lambda_k$, $k\in\mathds{N}$ such that 
$$\text{disc}_z(G_d)(\lambda_i)=0, \quad i=1, 2, \dots, k.$$
Now, we also have that for each $\lambda_i$ there are at most some $z_1^i, \dots, z_m^i$ such that 
$$G_d(\lambda_i, z_j^i)=0, \hspace{2em}i=1, \dots, k; \hspace{0.5em}j=1, \dots, m.$$ 
Define the set
$$F:=\{\lambda\in\mathds{C} : G_d(\lambda, z_j^i)=0\text{ for some }i=1, \dots, k; \hspace{0.5em}j=1, \dots, m\},$$
noting that $F$ is finite. By taking $\lambda_0\in i\mathds{R}\setminus F$ we have that $\text{disc}_z(G_d)(\lambda_0)\neq 0$, and therefore, there is some $z_0\in\mathds{C}$ such that $G_d(\lambda_0, z_0)=0$ and $\frac{\partial G_d}{\partial z}(\lambda_0, z_0)\neq 0$. 
Moreover, if $\lambda\in\mathds{C}$ is such that $G_d(\lambda, z_0)=0$ then $\lambda\notin F$, and therefore $\frac{\partial G_d}{\partial z}(\lambda, z_0)\neq 0$. The latter implies that there are some $\lambda_0\in i\mathds{R}$ and $z_0\in\mathds{C}$ such that $G_d(\lambda_0, z_0)=0$ and $\text{res}_{\lambda}\left(G_d, \frac{\partial G_d}{\partial z}\right)(z_0)\neq 0$. The conclusion follows for any $\rho>0$ such that $z_0\in\rho\mathds{D}$ by condition $(ii)$ on Theorem \ref{caos matriz semigrupos}.

\end{proof}

\begin{corollary}\label{corolario para racionales}
    Let $q, p_1, p_2, \dots, p_n$ be polynomials of degree at most $m$. Moreover, let $\rho>0$ be such that 
    $$\{z\in\mathds{C}: q(z)=0\}\subset\mathds{C}\setminus\rho\overline{\mathds{D}}.$$
 Define the function $P:\mathds{C}\times \overline{\rho\mathds{D}}\rightarrow\mathds C$ as $$P(\lambda, z):=p_1(z)+\dots+\lambda^{n-1}p_n(z)-\lambda^{n}q(z)=a_m(\lambda)z^m+\dots+a_1(\lambda)z+a_0(\lambda),$$
where $a_0, a_1, ..., a_m$ are polynomial of degree at most $n$ such that $\text{disc}_z(P)(\lambda)\nequiv 0$. Then the block matrix operator 
    $$A_d:= \begin{pmatrix}
0 & I & 0 & \dots \\
0 & 0 & I & \dots \\
\vdots & \vdots & \vdots & \ddots \\
q(\partial_x)^{-1}p_1(\partial_x) & q(\partial_x)^{-1}p_2(\partial_x) & \dots & q(\partial_x)^{-1}p_n(\partial_x) 
\end{pmatrix}$$
is a bounded operator on $(X_{\rho})^n$ that generates a Devaney sub-chaotic semigroup if and only if there are some $\lambda_0\in i\mathds{R}$ and $z_0\in\rho\mathds{D}$ such that $P(\lambda_0, z_0)=0$ and $\textnormal{disc}_z(P)(\lambda_0)\neq 0$. 
\end{corollary}
\begin{proof}
    We first observe that in this case the characteristic polynomial is such that
    $$G_d(\lambda, z)=\frac{P(\lambda, z)}{q(z)}.$$
    Since $q(z)\neq 0$ for all $z\in\rho\overline{\mathds{D}}$ we have that $G_d(\lambda, z)=0$ if and only if $P(\lambda, z)=0$. Moreover, it is easy to check that $\text{res}_{\lambda}(G_d, \frac{\partial G_d}{\partial z})(z)=0$ if and only if $\text{res}_{\lambda}(P, \frac{\partial P}{\partial z})(z)=0$. As a consequence, by Theorem \ref{caos matriz semigrupos} it suffices to prove that there are some $\lambda_0\in i\mathds{R}$ and $z_0\in\rho\mathds{D}$ such that $P(\lambda_0, z_0)=0$ and $\text{disc}_z(P)(\lambda_0)\neq 0$ if and only if there are some $\hat\lambda\in i\mathds{R}$ and $\hat z\in\rho\mathds{D}$ with $P(\hat\lambda, \hat z)=0$ and $\text{res}_{\lambda}{(P, \frac{\partial P}{\partial z})}(\hat z)\neq 0$. \\

    Assume that there are some $\lambda_0\in i\mathds{R}$ and $z_0\in\rho\mathds{D}$ such that $P(\lambda_0, z_0)=0$ and $\text{disc}_z{(P)}(\lambda_0)\neq 0$. As in the proof of corollary \ref{corolario para polinomios} there are at most $\lambda_1, ..., \lambda_N$ such that $\text{disc}_z{(P)}(\lambda_i)=0$ for $i=1, 2, .... N$. For each $\lambda_i$ there are at most some $z_1^i, ..., z_m^i$ such that $P(\lambda_i, z_i^j)=0$, so that we can define the set
    $$F_P:=\{\lambda\in\mathds{C}\text{ such that }P(\lambda, z_i^j)=0\text{ for some }i=1, ..., k; \hspace{0.5em}j=1,  ..., m\}.$$

    Now since $\text{disc}_z{(P)}(\lambda_0)\neq 0$ we have that $\frac{\partial P}{\partial z}(\lambda_0, z_0)\neq 0$. By the Implicit Function Theorem for holomorphic functions there is some open neighborhood $U$ of $\lambda_0$ and a holomorphic function $g:U\rightarrow\rho\mathds{D}$ such that $P(\lambda, g(\lambda))=0$ for all $\lambda\in U$. Since $F_P$ is finite, it follows that there is some $\hat\lambda\in U\cap i\mathds{R}$ such that $\hat\lambda\notin F_P$, and therefore $\text{res}_{\lambda}{(P, \frac{\partial P}{\partial z})}(g(\hat\lambda))\neq 0$, and taking $\hat z=g(\hat\lambda)$ the conclusion holds. \\

    Assume now that there are some $\hat\lambda\in i\mathds{R}$ and $\hat z\in \rho\mathds{D}$ such that $P(\hat\lambda, \hat z)=0$ and $\text{res}_{\lambda}{(P, \frac{\partial P}{\partial z})}(\hat z)\neq 0$. Then we have that $\frac{\partial P}{\partial z}(\hat\lambda, \hat z)\neq 0$ and again by the Implicit Function Theorem for holomorphic functions there is some open neighborhood $\hat U$ of $\hat\lambda$ and some holomorphic function $\hat g:\hat U\rightarrow \rho\mathds{D}$ such that $P(\lambda, \hat g(\lambda))=0$ for all $\lambda\in \hat U$. Since $\text{disc}_z(P)(\lambda)\nequiv 0$, there is at most $\lambda_1, \lambda_2, ..., \lambda_N$ such that $\text{disc}_{P(\lambda, \cdot)}(\lambda_i)=0$, $i=1, 2, ..., N$. Consequently, there is some $\lambda_0\in \hat U\cap i\mathds{R}$ such that $P(\lambda_0, \hat g(\lambda_0))=0$ and $\text{disc}_z(P)(\lambda_0)\neq 0$, and naming $z_0:=\hat g(\lambda_0)$ the conclusion follows. 
\end{proof}

We also highlight the following particular case. 

\begin{corollary}\label{corolario para racionales caso particular}
    Let $p_1, p_2, \dots, p_n$ be polynomials of degree at most $m$. Let $\alpha\in\mathds{C}$ and $\rho>0$ such that $|\alpha|\rho^k<1$. Define the function $P:\mathds{C}\times \overline{\rho\mathds{D}}\rightarrow\mathds C$ as $$P(\lambda, z):=p_1(z)+\dots+\lambda^{n-1}p_n(z)-\lambda^{n}(1-\alpha z^k)=a_l(\lambda)z^l+\dots+a_1(\lambda)z+a_0(\lambda),$$
where $a_1, a_2, ..., a_l$ are polynomial of degree at most $n$ and $l:=\max\{k, m\}$ such that $\text{disc}_z{(P)}(\lambda)\nequiv 0$. Then the block matrix operator 
    $$A_d:= \begin{pmatrix}
0 & I & 0 & \dots \\
0 & 0 & I & \dots \\
\vdots & \vdots & \vdots & \ddots \\
(I-\alpha\partial_{x}^k)^{-1}p_1(\partial_x) & (I-\alpha\partial_{x}^k)^{-1}p_2(\partial_x) & \dots & (I-\alpha\partial_{x}^k)^{-1}p_n(\partial_x) 
\end{pmatrix}$$
is a bounded operator on $(X_{\rho})^n$ that generates a Deveney sub-chaotic semigroup if and only if there are some $\lambda_0\in i\mathds{R}$ and $z_0\in\rho\mathds{D}$ such that $P(\lambda_0, z_0)=0$ and $\text{disc}_z{(P)}(\lambda_0)\neq 0$. 
\end{corollary}

\begin{remark}\label{condicion suficiente caos caso particular} If there are some $\lambda_0\in i\mathds{R}$, $0<r\leq \rho$ and $j\in\{1, 2, ..., l\}$ such that $\textnormal{disc}_z{(P)}(\lambda_0)\neq 0$ and 
$$|a_0(\lambda_0)|+ \sum_{i\neq j}|a_i(\lambda_0)|r^i<|a_j(\lambda_0)|r^j,$$
by Rouche's Theorem we have that $P(\lambda_0, \cdot)$ has at least $j$ distinct roots on $r\mathds{D}$. Hence by Corollary \ref{corolario para racionales} the block operator matrix $A_d$ generates a sub-chaotic semigroup on $X_{\rho}^n$.  
\end{remark}
 As in the first order case, we now apply the previous results to the analysis of the associated solution semigroups in equations \eqref{caso ecuacion polinomial} and \eqref{caso ecuación racional}.

 For the first equation, we can write system \eqref{sistema ecuaciones caso polinomico} in the form of the following abstract Cauchy problem: 
 \begin{equation}\label{Cauchy orden n caso polinomial}
     \frac{\partial}{\partial t}\begin{pmatrix}
         u_1\\
         u_2\\
         \vdots\\
         u_n
     \end{pmatrix}=\begin{pmatrix}
0 & I & 0 & \dots \\
0 & 0 & I & \dots \\
\vdots & \vdots & \vdots & \ddots \\
p_1(\partial_x) & p_2(\partial_x) & \dots & p_n(\partial_x) 
\end{pmatrix}\begin{pmatrix}
    u_1\\
    u_2\\
    \vdots\\
    u_n
\end{pmatrix},\quad \begin{pmatrix}
    u_1(0)\\
    u_2(0)\\
    \vdots\\
    u_n(0)
\end{pmatrix}\in X_{\rho}^n,\\
 \end{equation}
 where $p_i(\partial_x)=-\sum_{j=0}^ma_{i-1, j}\partial_x^j$ for $i=1, 2, ..., n$. The following result is a direct consequence of corollary \ref{corolario para polinomios}. 
 \begin{corollary}
     Let a PDE be such that its highest-order time derivative is of order $n\in\mathds{N}$ and appears without mixed spatial derivatives, as in the form of equation \eqref{caso ecuacion polinomial}. Then there is some $\rho_0>0$ such that the projection family of operators associated to the solution semigroup is Devaney chaotic on $X_{\rho}^n$ for all $\rho>\rho_0$. 
 \end{corollary}

 Similarly, for equation \eqref{caso ecuación racional} we define the polynomials $p_i(z)=-\sum_{j=0}^ma_{i-1, j}z^j$ for $i=1, 2, ..., n$ and $q(z)=I+\sum_{j=1}^m a_{n, j}z^j$. It follows that for all $\rho>0$ such that 
$$\{z\in\mathds{C}: q(z)=0\}\subset \mathds{C}\setminus\overline{\rho\mathds{D}},$$
the operator $q(\partial_x)^{-1}=(1/q)(\partial_x)$ is well defined on $X_{\rho}$. Hence we can express system \eqref{sistema ecuaciones caso racional} as in the following abstract Cauchy problem: 
\begin{equation}\label{Cauchy orden n caso racional}
     \frac{\partial}{\partial t}\begin{pmatrix}
         u_1\\
         u_2\\
         \vdots\\
         u_n
     \end{pmatrix}=\begin{pmatrix}
0 & I & 0 & \dots \\
0 & 0 & I & \dots \\
\vdots & \vdots & \vdots & \ddots \\
(p_1/q)(\partial_x) & (p_2/q)(\partial_x) & \dots & (p_n/q)(\partial_x) 
\end{pmatrix}\begin{pmatrix}
    u_1\\
    u_2\\
    \vdots\\
    u_n
\end{pmatrix},\quad \begin{pmatrix}
    u_1(0)\\
    u_2(0)\\
    \vdots\\
    u_n(0)
\end{pmatrix}\in X_{\rho}^n.\\
 \end{equation}
 Applying Corollary \ref{corolario para racionales} we get the following result. 
\begin{corollary}
    Let a PDE be as in the form of equation \eqref{caso ecuación racional} and let $\rho>0$ be such that 
    $$\{z\in\mathds{C}: q(z)=0\}\subset\mathds{C}\setminus\rho\overline{\mathds{D}}.$$
 Define the function $P:\mathds{C}\times \overline{\rho\mathds{D}}\rightarrow\mathds C$ as $$P(\lambda, z):=p_1(z)+\dots+\lambda^{n-1}p_n(z)-\lambda^{n}q(z)=a_m(\lambda)z^m+\dots+a_1(\lambda)z+a_0(\lambda),$$
where $a_0, a_1, ..., a_m$ are polynomial of degree at most $n$. Then the projection family of operators associated to the solution semigroup is chaotic if and only if there are some $\lambda_0\in i\mathds{R}$ and $z_0\in\rho\mathds{D}$ such that $P(\lambda_0, z_0)=0$ and $\text{disc}_{P(\lambda, \cdot)}(\lambda_0)\neq 0$.  
\end{corollary}
\begin{remark} If $q(\partial_x):=I-\alpha \partial_x^k$ for some $\alpha\in\mathds{C}$ and $k\in\mathds{N}$, applying Corollary \ref{corolario para racionales caso particular} we characterize chaos for all $0<\rho<\frac{1}{|\alpha|^{1/k}}$. Moreover, Remark \ref{condicion suficiente caos caso particular} provides a sufficient condition in this particular case. 
\end{remark}
\section{Chaos in spatial semidiscretizations}
In this section we establish a connection between the analysis of PDEs in Herzog spaces of analytic functions and their spatial semi-discretization. For a fixed discretization step $h>0$, we consider the following finite differences: 
\begin{align*}
    &\partial_xu(t, nh)\approx\Delta_h^+u(t, n):=\frac{u(t, n+1)-u(t, n)}{h}\\
    &\partial_xu(t, nh)\approx\Delta_h^-u(t, n):=\frac{u(t, n)-u(t, n-1)}{h},
\end{align*}
which correspond to the forward and backward discretization schemes respectively. Furthermore, assuming that for any $t>0$, $u(t, \cdot)\in\ell^p(\mathds{N}_0)$ (for $1\leq p<\infty$) or $c_0(\mathds{N}_0)$, these operators can be expressed as: 
$$\Delta_h^+=\frac{B-I}{h} \hspace{1em}\text{ and }\hspace{1em} \Delta_h^-=\frac{I-F}{h},$$ where $F$ and $B$ are the forward and backward shift operators. As a consequence, we can characterize the spectrum of these operators: 
\begin{equation}\label{espectro hacia adelante}
    \sigma(\Delta_h^+)=\overline{B\left(\frac{-1}{h}, \frac{1}{h}\right)}, \quad \sigma_P(\Delta_h^+)=B\left(\frac{-1}{h}, \frac{1}{h}\right);
\end{equation}
\begin{equation}\label{espectro hacia atrás}
    \sigma(\Delta_h^-)=\overline{B\left(\frac{1}{h}, \frac{1}{h}\right)}, \quad \sigma_P(\Delta_h^{-})=\emptyset.
\end{equation}
Hence, for $h>0$, we will consider the spatial semi-discretization of systems \eqref{sistema ecuaciones caso polinomico} and \eqref{sistema ecuaciones caso racional}, where $\Delta_h$ represents either the forward $(\Delta_h^+)$ or backward $(\Delta_h^-)$ difference scheme. 
\begin{equation}\label{sistema discreto ecuaciones caso polinomico}
\begin{cases}
    &u_1=u\\
    &\partial_t u_1=u_2\\
    &\vdots\\
    &\partial_t u_{n-1}=u_n\\
    &\partial_tu_n=-\sum_{j=0}^ma_{n-1, j}\Delta_h^ju_n-...-\sum_{j=0}^ma_{0, j}\Delta_h^j u_1
\end{cases},
\end{equation} 
\begin{equation}\label{sistema discreto ecuaciones caso racional}
\begin{cases}
    &u_1=u\\
    &\partial_t u_1=u_2\\
    &\vdots\\
    &\partial_t u_{n-1}=u_n\\
&\partial_t(I+\sum_{j=1}^ma_{n,j}\Delta_h^j)u_n=-\sum_{j=0}^ma_{n-1, j}\Delta_h^ju_n-...-\sum_{j=0}^ma_{0, j}\Delta_h^j u_1
\end{cases}.
\end{equation}

\subsection{Backward difference scheme}
Consider the backward difference scheme $\Delta_h^-$ as an operator on $X=\ell^p(\mathds{N}_0)$ or $c_0(\mathds{N}_0)$. Then using functional calculus and the fact that $F^*=B$ we have that the adjoint operator $(\Delta_h^-)^*$ on $X^*=\ell^q(\mathds{{N}}_0), 1\leq q\leq \infty,$ is such that 
$$(\Delta_h^-)^*=\frac{I-B}{h}.$$
As a consequence, we have that 
\begin{equation}\label{ec espectro adjunto 1}
    \sigma_P((\Delta_h^-)^*)=B\left(\frac{1}{h}, \frac{1}{h}\right).
\end{equation}
Now define the polynomial functions $p_1, ..., p_n, q$ as 
$$p_i(z)=-\sum_{j=0}^m a_{i-1, j}z^j, i=1, 2, ..., n; \quad q(z)=1+\sum_{j=1}^m a_{n, j}z^j.$$
Then if $q$ is such that 
$$\{z\in \mathds{C}: q(z)=0\}\subset \mathds{C}\setminus \overline{B(1/h, 1/h)},$$
we can write systems \eqref{sistema discreto ecuaciones caso polinomico} and \eqref{sistema discreto ecuaciones caso racional} as
\begin{equation}\label{Cauchy orden n disretizacion hacia atrás}
     \frac{\partial}{\partial t}\begin{pmatrix}
         u_1\\
         u_2\\
         \vdots\\
         u_n
     \end{pmatrix}=\begin{pmatrix}
0 & I & 0 & \dots \\
0 & 0 & I & \dots \\
\vdots & \vdots & \vdots & \ddots \\
(p_1/q)(\Delta_h^-) & (p_2/q)(\Delta_h^-) & \dots & (p_n/q)(\Delta_h^-) 
\end{pmatrix}\begin{pmatrix}
    u_1\\
    u_2\\
    \vdots\\
    u_n
\end{pmatrix},\quad \begin{pmatrix}
    u_1(0)\\
    u_2(0)\\
    \vdots\\
    u_n(0)
\end{pmatrix}\in X^n.\\
 \end{equation}
We note that, in this case, we are considering system \eqref{sistema discreto ecuaciones caso polinomico} to be a particular case of \eqref{sistema discreto ecuaciones caso racional} with $q=1$. We can define the block operator matrix $A_h^-:X^n\rightarrow X^n$ as the coefficient matrix of system \eqref{Cauchy orden n disretizacion hacia atrás}, so that the solution can be expressed in terms of the uniformly continuous semigroup $(e^{tA_h^-})_{t\geq 0}$. The following Theorem studies the dynamics in terms of chaos of this solution semigroup.
 
 \begin{theorem}\label{backward semidiscretization} The semigroup $(e^{tA_h^-})_{t\geq 0}$ fails to be chaotic on $X^n$ ($X=\ell^ p(\mathds{N}_0)$ or $c_0(\mathds{N}_0))$ for any $h>0$. 
 \end{theorem}
 \begin{proof}
     We show that the adjoint operator $(A_h^-)^*$ has eigenvalues, which proves that the semigroup cannot be hypercyclic \cite[Lemma 7.14]{Alfred}. 
     
     Given a block-operator matrix
     $$A=\begin{pmatrix}
         A_{11}&\dots &A_{1n}\\
         \vdots & \ddots &\vdots\\
         A_{n1}&\dots &A_{nn}
     \end{pmatrix}, \quad A_{ij}: X\rightarrow X, \text{ for all }i, j=1, ..., n,$$
     we can compute the adjoint operator $A^*: (X^n)^*$ as
     $$A^*=\begin{pmatrix}
         A_{11}^*&\dots &A_{n1}^*\\
         \vdots & \ddots &\vdots\\
         A_{1n}^*&\dots &A_{nn}^*
     \end{pmatrix}.$$
     Hence, using functional calculus, the adjoint operator of the block-operator matrix $A_h^-$ is as follows: 
     \begin{align*}
         &(A_h^-)^*=\begin{pmatrix}
0 & 0& \dots &(p_1/q)(\Delta_h^-)^*\\
I & 0 & \dots & (p_2/q)(\Delta_h^-)^* \\
\vdots & \ddots & \vdots & \vdots \\
 0& \dots & I & (p_n/q)(\Delta_h^-)^* 
\end{pmatrix}=\begin{pmatrix}
0 & 0& \dots &(p_1/q)((\Delta_h^-)^*)\\
I & 0 & \dots & (p_2/q)((\Delta_h^-)^*) \\
\vdots & \ddots & \vdots & \vdots \\
 0& \dots & I & (p_n/q)((\Delta_h^-)^*) 
\end{pmatrix}\\
&=\begin{pmatrix}
0 & 0& \dots &(p_1/q)(\frac{1-B}{h})\\
I & 0 & \dots & (p_2/q)(\frac{1-B}{h}) \\
\vdots & \ddots & \vdots & \vdots \\
 0& \dots & I & (p_n/q)(\frac{1-B}{h}) 
\end{pmatrix}=\begin{pmatrix}
0 & 0& \dots &g_{h, 1}(B)\\
I & 0 & \dots & g_{h, 2}(B) \\
\vdots & \ddots & \vdots & \vdots \\
 0& \dots & I & g_{h, n}(B) 
\end{pmatrix},
\end{align*}
where $g_{h, i}(z):=\frac{p_i(\frac{1-z}{h})}{q(\frac{1-z}{h})}$ is holomorphic on a neighborhood of $\sigma(B)$ for all $i=1, ..., n$. 

Now define for $\lambda\neq 0$ the operator $G_{h, \lambda}: X^*\rightarrow X^*$ as
$$G_{h, \lambda}:=\left(\sum_{i=1}^{n} \frac{g_{h, i}(B)}{\lambda^{n-i}}\right)-\lambda I.$$

 We claim that for any $\lambda\neq 0$ such that $0\in\sigma_P(G_{h, \lambda})$ then $\lambda\in \sigma_p((A_h^-)^*)$. Indeed, take $u_{0}\in X^*$ such that $G_{h, \lambda}u_0=0$ and define the vector $v_{\lambda}\in (X^*)^n$ as
$$v_{\lambda}:=\left(\frac{g_{h, 1}(B)}{\lambda}u_0,\left(\frac{g_{h, 1}(B)}{\lambda^2}+\frac{g_{h, 2}(B)}{\lambda}\right)u_0 ,  ..., \sum_{i=1}^{n-1}\frac{g_{h, i}(B)}{\lambda^{n-i}}u_0, u_0\right).$$
Then we can compute 
$$(A_h^-)^*v_{\lambda}=\begin{pmatrix}
    (g_{h, 1}(B))u_0\\
    \left(\frac{g_{h, 1}(B)}{\lambda}+g_{h, 2}(B)\right)u_0\\
    \vdots
    \\
    \sum_{i=1}^{n} \frac{g_{h, i}(B)}{\lambda^{n-i}}u_0
\end{pmatrix}=\begin{pmatrix}
    (g_{h, 1}(B))u_0\\
    \left(\frac{g_{h, 1}(B)}{\lambda}+g_{h, 2}(B)\right)u_0\\
    \vdots
    \\
    G_{h, \lambda}u_0+\lambda u_0
\end{pmatrix}=\lambda v_{\lambda},$$
and the claim holds. 

Now the conclusion follows by proving that for any $h>0$ we can find $\lambda\neq 0$ such that $0\in\sigma_P(G_{h, \lambda})$. Indeed, for any $\lambda\neq 0$, by the Spectral Mapping Theorem we have: 
$$\sigma_P(G_{h, \lambda})=\left\{\sum_{i=1}^{n}\frac{g_{h, i}(z)}{\lambda^{n-i}}-\lambda: z\in\mathds{D}\right\}. $$
Now since there is at least some polynomial $p_i$, $i\in\{1, 2, ..., n\}$, such that $p_i$ is nonconstant, there is some $z_0\in\mathds{D}$ such that $g_{h, i}(z_0)\neq 0$. Since $\lambda\neq 0$, we have that 
$$\sum_{i=1}^{n}\frac{g_{h, i}(z_0)}{\lambda^{n-i}}-\lambda=0\quad\text{ if and only if }\quad 
\sum_{i=1}^{n}g_{h, i}(z_0)\lambda^{i}-\lambda^{n+1}=0.$$
Finally, by the Fundamental Theorem of Algebra and the fact that $g_{h, i}(z_0)\neq 0$ it follows that there is some $\lambda_0\neq 0$ such that $\sum_{i=1}^{n}g_{h, i}(z_0)\lambda_0^{i}-\lambda_0^{n+1}=0$ and the proof holds. 
 \end{proof}
\subsection{Forward difference operator}
Let us now focus on the forward difference scheme $\Delta_h^+$ as an operator on $X=\ell^p(\mathds{N}_0)$ of $c_0(\mathds{N}_0)$. We prove that, contrary of what happens with the backward difference operator, in this case, chaos arises as a consequence of \ref{caos matriz A}. 
We first note that, in view that $\Delta_h^+=\frac{B-I}{h}$ and of its spectrum given in \eqref{espectro hacia adelante}, the results of the previous section for the operator $\partial_x$ can be naturally adapted to the present framework. 

Define $G_h: \mathds{C}\times \overline{B(-1/h, 1/h)}\rightarrow \mathds{C}$ as: 
$$G_h(\lambda, z):=f_1(z)+\lambda f_2(z)+\dots+\lambda^{n-1}f_n(z)-\lambda^n,$$
then we have the following result. 
\begin{theorem}\label{caos matriz semigrupo hacia adelante}
    Let $h>0$ and $f_1, f_2, ..., f_n$ be holomorphic functions on $\overline{B(-1/h, 1/h)}$ such that there exists some $\Tilde{z}\in B(-1/h, 1/h)$ with $\textnormal{res}_{\lambda}{\left(G_h, \frac{\partial G_h}{\partial_z}\right)}(\Tilde{z})\neq 0$. 
    Then the following assertions are equivalent: 
    \begin{enumerate}[(i)]
        \item The block operator matrix:
        $$A_h^+=\begin{pmatrix}
0 & I & 0 & \dots \\
0 & 0 & I & \dots \\
\vdots & \vdots & \vdots & \ddots \\
f_1(\Delta_h^+) & f_2(\Delta_h^+) & \dots & f_n(\Delta_h^+) 
\end{pmatrix}$$
generates a sub-chaotic semigroup $(e^{tA_h^+})_{t\geq 0}$ on $X^n$ such that the projection family $(\pi(e^{tA_h^+}): X^n\rightarrow X)_{t\geq 0}$ is chaotic.  \\
\item There are some $\lambda_0\in i\mathds{R}$, $z_0\in B(-1/h, 1/h)$ such that $G_h(\lambda_0, z_0)=0$ and $\textnormal{res}_{\lambda}{\left(G_h, \frac{\partial G_h}{\partial_z}\right)}(z_0)\neq 0$.\\
\item Every autonomous discretization of $(e^{tA_h^+})_{t\geq 0}$ is Devaney sub-chaotic on $X^n$. \\
\item Some autonomous discretization of $(e^{tA_h^+})_{t\geq 0}$ is Devaney sub-chaotic on $X^n$. 
    \end{enumerate}
\end{theorem}
\begin{proof}
    The proof is entirely analogous to that of Theorem \ref{caos matriz semigrupos}.
\end{proof}
\begin{corollary}\label{corolario para polinomios hacia adelante} Let $p_1, p_2, \dots, p_n$ be polynomials such that at least one of them is nonconstant. Then the following assertions are equivalent: 
\begin{enumerate}[(i)]
    \item There is some $h_0>0$ such that the semigroup $(T_t)_{t\geq 0}$ generated by the block matrix operator 
    $$A_h^+:= \begin{pmatrix}
0 & I & 0 & \dots \\
0 & 0 & I & \dots \\
\vdots & \vdots & \vdots & \ddots \\
p_1(\Delta_h^+) & p_2(\Delta_h^+) & \dots & p_n(\Delta_h^+) 
\end{pmatrix}$$
is Deveney sub-chaotic on $X^n$ for all $0<h<h_0$ and the projection family $(T_t: X^n\rightarrow X)_{t\geq 0}$ generated by $A_h^+$ is chaotic. 
\item There are some $\lambda_0\in i\mathds{R}$ and $z_0\in\mathds{C}$ with $\textnormal{Re}(z_0)<0$ such that $G_h(\lambda_0, z_0)=0$ and $\textnormal{disc}_z(G_h)(\lambda_0)\neq 0.$
\end{enumerate}
\end{corollary}
\begin{proof}
    In a similar way as in the proof of Corollary \ref{corolario para racionales}, we can show that there are some $\lambda_0\in i\mathds{R}$ and $z_0$ with $\textnormal{Re}(z_0)<0$ such that $G_h(\lambda_0, z_0)=0$ and $\textnormal{disc}_z(G_h)(\lambda_0)\neq 0$ if and only if there are some some $\hat{\lambda}\in i\mathds{R}$ and $\hat{z}$ with $\textnormal{Re}(\hat{z})<0$ such that $G_h(\hat{\lambda}, \hat{z})=0$ and $\textnormal{res}_{\lambda}{\left(G_h, \frac{\partial G_h}{\partial z}\right)}(\hat{z})\neq 0$. 

    Finally, we notice that given $z_0$ with $\textnormal{Re}(z_0)<0$, there is some $h_0$ such that $z_0\in B(-1/h, 1/h)$ for all $0<h<h_0$, and the conclusion holds by Theorem \ref{caos matriz semigrupo hacia adelante}.
\end{proof}
\begin{corollary}\label{corolario para racionales hacia adelante}
    Let $q, p_1, p_2, \dots, p_n$ be polynomials of degree at most $m$. Moreover, let $h>0$ be such that 
    $$\{z\in\mathds{C}: q(z)=0\}\subset\mathds{C}\setminus\overline{B(-1/h, 1/h)}.$$
 Define the function $P_h:\mathds{C}\times \overline{B(-1/h, 1/h)}\rightarrow\mathds C$ as $$P_h(\lambda, z):=p_1(z)+\dots+\lambda^{n-1}p_n(z)-\lambda^{n}q(z)=a_m(\lambda)z^m+\dots+a_1(\lambda)z+a_0(\lambda),$$
where $a_0, a_1, ..., a_m$ are polynomial of degree at most $n$ such that $\text{disc}_z(P_h)(\lambda)\nequiv 0$. Then the block matrix operator 
    $$A_h^+:= \begin{pmatrix}
0 & I & 0 & \dots \\
0 & 0 & I & \dots \\
\vdots & \vdots & \vdots & \ddots \\
(p_1/q)(\Delta_h^+) & (p_2/q)(\Delta_h^+) & \dots & (p_n/q)(\Delta_h^+) 
\end{pmatrix}$$
is a bounded operator on $X^n$ that generates a Deveney sub-chaotic semigroup if and only if there are some $\lambda_0\in i\mathds{R}$ and $z_0\in B(-1/h, 1/h)$ such that $P_h(\lambda_0, z_0)=0$ and $\textnormal{disc}_z(P_h)(\lambda_0)\neq 0$. 
\end{corollary}

Now as in the backward difference scheme, we take the polynomial functions $p_1, ..., p_n, q$ and assume that $q$ is such that 
$$\{z\in \mathds{C}: q(z)=0\}\subset \mathds{C}\setminus \overline{B(-1/h, 1/h)}.$$
Then we write systems \eqref{sistema discreto ecuaciones caso polinomico} and \eqref{sistema discreto ecuaciones caso racional} as
\begin{equation}\label{Cauchy orden n disretizacion hacia adelante}
     \frac{d}{d t}\begin{pmatrix}
         u_1\\
         u_2\\
         \vdots\\
         u_n
     \end{pmatrix}=\begin{pmatrix}
0 & I & 0 & \dots \\
0 & 0 & I & \dots \\
\vdots & \vdots & \vdots & \ddots \\
(p_1/q)(\Delta_h^+) & (p_2/q)(\Delta_h^+) & \dots & (p_n/q)(\Delta_h^+) 
\end{pmatrix}\begin{pmatrix}
    u_1\\
    u_2\\
    \vdots\\
    u_n
\end{pmatrix},\quad \begin{pmatrix}
    u_1(0)\\
    u_2(0)\\
    \vdots\\
    u_n(0)
\end{pmatrix}\in X^n.\\
 \end{equation}
As a consequence of Theorem \ref{caos matriz semigrupo hacia adelante} and Corollaries \ref{corolario para polinomios hacia adelante} and \ref{corolario para racionales hacia adelante} we have the following results. 
\begin{corollary} Let a PDE be such that its highest-order time derivative is of order $n\in\mathds{N}$ and appears without mixed spatial derivatives as in \eqref{caso ecuacion polinomial}
    . Then the following assertions are equivalent: 
    \begin{enumerate}[(i)]
        \item There is some $h_0>0$ such that the solution of the forward semidiscretization in space of the equation of step $h>0$ is Devaney chaotic on $X=\ell^p(\mathds{N}_0)$ or $c_0(\mathds{N}_0)$ for all $0<h<h_0$. 
        \item There are some $\lambda_0\in i\mathds{R}$ and $z_0$ such that $\textnormal{Re}(z_0)<0$, $G_h(\lambda_0, z_0)=0$ and $\textnormal{disc}_{G_h}(\lambda_0)\neq 0$.
    \end{enumerate}
\end{corollary}
\begin{corollary}
    Let a PDE be as in the form of equation \eqref{caso ecuación racional} and let $h>0$ be such that 
    $$\{z\in\mathds{C}: q(z)=0\}\subset\mathds{C}\setminus \overline{B(-1/h, 1/h)}.$$
 Define the function $P_h:\mathds{C}\times \overline{B(-1/h, 1/h)}\rightarrow\mathds C$ as $$P_h(\lambda, z):=p_1(z)+\dots+\lambda^{n-1}p_n(z)-\lambda^{n}q(z)=a_m(\lambda)z^m+\dots+a_1(\lambda)z+a_0(\lambda),$$
where $a_0, a_1, ..., a_m$ are polynomial of degree at most $n$. Then the projection family associated to the solution semigroup of equation \eqref{Cauchy orden n disretizacion hacia adelante} is chaotic if and only if there are some $\lambda_0\in i\mathds{R}$ and $z_0\in B(-1/h, 1/h)$ such that $P_h(\lambda_0, z_0)=0$ and $\text{disc}_z(P_h)(\lambda_0)\neq 0$.  
\end{corollary}

Finally, we observe a mutual isometric relationship between continuous and semidiscretized problems.

\begin{theorem}\label{isometrias}
Given an abstract Cauchy problem as in \eqref{Cauchy orden n caso polinomial} or \eqref{Cauchy orden n caso racional} on a Herzog space with weight $\rho>0$, there exists a PDE of the form \eqref{caso ecuacion polinomial} or \eqref{caso ecuación racional} whose forward semidiscretization with step size $h=1/\rho$, as in \eqref{Cauchy orden n disretizacion hacia adelante}, is isometric to it.

Conversely, given a forward semidiscretization of an abstract Cauchy problem with spatial step size $h>0$, there exists an equivalent continuous Cauchy problem in terms of spatial derivatives on a Herzog space with weight $\rho=1/h$.
\end{theorem}
\begin{proof}
Indeed, take $\rho=1/h$, $g_h(z):=z+1/h$ and define $\Phi_\rho : X_\rho \rightarrow c_0$ as the isometric isomorphism between $X_\rho$ and $c_0(\mathds{N}_0)$. Recall that under this isometry $\partial_x\simeq \rho B$. Then since $\Delta_h^+=\frac{B-I}{h}$, for any function $f$ holomorphic on $\overline{B(0, \rho)}$ we have:  
$$\Phi_{\rho}\circ f(\partial_x)=(f\circ g_h)(\Delta^+_h)\circ \Phi_{\rho},$$
that is, the following diagram
$$
\begin{CD}
X_\rho @> f(\partial_x) >> X_\rho \\
@V \Phi_\rho VV @VV \Phi_\rho V \\
c_0 @>> (f\circ g_h)(\Delta_h^+) > c_0
\end{CD}
$$
commutes. Hence if $p_1, ..., p_n, q$ are polynomial functions defining a block operator matrix $A_d$ as in \eqref{Cauchy orden n caso polinomial} or \eqref{Cauchy orden n caso racional} on a Herzog space of weight $\rho>0$, then the functions $(p_1/q)\circ g_h, ..., (p_n/q)\circ g_h$ define an isometric operator $A_h^+$ of the form of equation \eqref{Cauchy orden n disretizacion hacia adelante}.\\

Conversely, taking $g_h^{-1}(z)=z-1/h$ we have that for any function $f$ holomorphic on $\overline{B(-1/h, 1/h)}$: 
$$\Phi_{\rho}^{-1}f(\Delta_h^+)=(f\circ g_{h}^{-1})(\partial_x)\circ \Phi_\rho^{-1},$$
so that 
$$
\begin{CD}
c_0 @> f(\Delta_h^+) >> c_0 \\
@V \Phi_\rho^{-1} VV @VV \Phi_\rho^{-1} V \\
X_{\rho} @>> (f\circ g_h^{-1})(\partial_x) > c_0
\end{CD}
$$
commutes. And therefore for every block operator matrix $A_h^+$ as in \eqref{Cauchy orden n disretizacion hacia adelante}, there is an isometric operator $A_d$ on a Herzog space $X_{\rho}^n$ with $\rho=1/h$. 
\end{proof}
\section{Some applications}
In this section, we will apply the previous results on two different models: the fourth order Moore-Gibson-Thompson equation and the viscous van
Wijngaarden–Eringen equation. Although the conditions for chaos were studied in \cite{lizamamurilloMGT} and \cite{conejerol_lizama_murillo2016vanwjingaarne}, we prove that no condition on the coefficients is needed if we want to ensure chaos just for the solution of the equation and not the solution and its derivatives. We also provide conditions on the spatial step of the discretization for Devaney chaos in the semidiscrete models. 
\subsection{Fourth-order Moore-Gibson-Thompson equation}
Let us consider the model: 
\begin{equation}\label{Moore-Gibson-Thompson}
    \frac{\partial^4u}{\partial t^4}+\alpha\frac{\partial ^3u}{\partial t^3}+\beta\frac{\partial ^2 u}{\partial t^2}-\gamma \frac{\partial ^4 u}{\partial t^2\partial x^2}-\delta \frac{\partial ^3 u}{\partial t\partial x^2}-r\frac{\partial ^2u}{\partial x^2}=0,\quad \alpha, \beta, \gamma, \delta, r>0.
\end{equation}
Following a similar proceeding as in system \eqref{sistema ecuaciones caso polinomico}, we can write the fourth order Moore-Gibson-Thompson equation as:
\begin{equation}\label{Cauchy Moore-Gibson-Thompson}
     \frac{\partial}{\partial t}\begin{pmatrix}
         u_1\\
         u_2\\
         u_3\\
         u_4
     \end{pmatrix}=\begin{pmatrix}
0 & I & 0 & 0 \\
0 & 0 & I & 0 \\
0 & 0 & 0 & I \\
r\partial^2_{x} & \delta\partial^2_x & \gamma\partial^2_x-\beta I & -\alpha I  
\end{pmatrix}\begin{pmatrix}
    u_1\\
    u_2\\
    u_3\\
    u_n
\end{pmatrix},\quad \begin{pmatrix}
    u_1(0)\\
    u_2(0)\\
    u_3(0)\\
    u_4(0)
\end{pmatrix}\in X_{\rho}^4.\\
 \end{equation}
 In \cite{lizamamurilloMGT}, the authors provided a sufficient condition in terms of the coefficients of the equation that ensure chaos for the solution semigroup on $X_\rho^4$ for some big enough weight $\rho>0$. The following Theorem states that no condition is needed if we want to ensure chaos strictly on the solution of the equation and not on the solution and its derivatives. 
\begin{theorem}\label{caos MGT} For any $\alpha, \beta, \delta, r, \gamma>0$ there is some $\rho_0>0$ such that the projection family associated to the solution semigroup of the Moore-Gibson-Thompson equation is Devaney chaotic on $X_\rho$ for all $\rho>\rho_0$.  
\begin{proof}
    In order to prove the result we apply Corollary \ref{corolario para polinomios} identifying the block operator matrix
    $$A_d=\begin{pmatrix}
0 & I & 0 & 0 \\
0 & 0 & I & 0 \\
0 & 0 & 0 & I \\
r\partial^2_{x} & \delta\partial^2_x & \gamma\partial^2_x-\beta I & -\alpha I  
\end{pmatrix}$$
and the polynomials $p_1(z):=r z^2$, $p_2(z):=\delta z^2$, $p_3(z):=\gamma z^2-\beta$, $p_4(z):=-\alpha$. As a consequence we have: 
$$G_d(\lambda, z)=rz^2+\lambda(\delta z^2)+\lambda^2(\gamma z^2-\beta)-\lambda^3\alpha-\lambda^4=z^2(r+\lambda\delta+\lambda^2\gamma)-(\lambda^4+\lambda^3\alpha+\lambda^2\beta).$$
And therefore for each fixed $\lambda\in\mathds{C}$ we have that $G_d(\lambda, z)=0$ if and only if 
$$z=\pm \sqrt{\frac{\lambda^4+\lambda^3\alpha+\lambda^2\beta}{r+\lambda\delta+\lambda^2\gamma}}.$$
Considering for instance $\lambda=1$ we have that $z=\pm\sqrt{\frac{1+\alpha+\beta}{r+\delta+\gamma}}$, which indicates that $\text{disc}_z{(G_d)}(1)\neq 0$. The conclusion holds as an immediate consequence of  Corollary \ref{corolario para polinomios}.
\end{proof}
\end{theorem}
\begin{theorem}
    The following assertions hold: 
    \begin{enumerate}[(i)]
        \item The backward semidiscretization of the Moore-Gibson-Thompson equation is not Devaney chaotic on $\ell^p(\mathds{N}_0)$, $1\leq p<\infty$ or $c_0(\mathds{N}_0)$. 
        \item For all $\alpha, \beta, \delta, r, \gamma>0$ such that $\delta\neq \alpha\gamma$ or $\alpha r\neq \beta \delta$, there is some $h_0>0$ such that the projection family of the forward semidiscretization of the Moore-Gibson-Thompson equation is Devaney chaotic on $\ell^p(\mathds{N}_0)$ or $c_0(\mathds{N}_0)$ for each $h<h_0$.
        \item For all $\alpha, \beta, \delta, r, \gamma>0$ such that $\delta= \alpha\gamma$ and $\alpha r= \beta \delta$, the forward semidiscretization of the Moore-Gibson-Thompson equation is not Devaney chaotic on $\ell^p(\mathds{N}_0)$, $1\leq p<\infty$ or $c_0(\mathds{N}_0)$.
    \end{enumerate}
\end{theorem}
\begin{proof}
    Assertion $(i)$ is a direct consequence of Theorem \eqref{backward semidiscretization}. \\
    For assertions $(ii)$ and $(iii)$ we apply corollary \ref{corolario para polinomios} identifying 
    $$A_h^+=\begin{pmatrix}
0 & I & 0 & 0 \\
0 & 0 & I & 0 \\
0 & 0 & 0 & I \\
r(\Delta_h^+)^2 & \delta(\Delta_h^+)^2 & \gamma(\Delta_h^+)^2-\beta I & -\alpha I  
\end{pmatrix}$$
and the polynomials $p_1(z):=r z^2$, $p_2(z):=\delta z^2$, $p_3(z):=\gamma z^2-\beta$, $p_4(z):=-\alpha$. As a consequence, we compute the characteristic polynomial: 
$$G_h(\lambda, z)=rz^2+\lambda(\delta z^2)+\lambda^2(\gamma z^2-\beta)-\lambda^3\alpha-\lambda^4=z^2(r+\lambda\delta+\lambda^2\gamma)-(\lambda^4+\lambda^3\alpha+\lambda^2\beta).$$
As in the proof of \ref{caos MGT} we have that, for each fixed $\lambda\in\mathds{C}$, $G_h(\lambda, z)=0$ if and only if 
$$z=\pm \sqrt{\frac{\lambda^4+\lambda^3\alpha+\lambda^2\beta}{r+\lambda\delta+\lambda^2\gamma}}.$$
Now take $\lambda=\omega i$ for some $w\in\mathds{R}$. Then we obtain: 
\begin{align*}
&z^2(\omega)= \frac{\omega^4 - \beta\omega^2 - i\alpha\omega^3}{r - \gamma\omega^2 + i\delta\omega} \\
&= \frac{((\omega^4 - \beta\omega^2) - i\alpha\omega^3) \cdot ((r - \gamma\omega^2) - i\delta\omega)}{(r - \gamma\omega^2)^2 + (\delta\omega)^2} \\
&= \frac{\left((\omega^4 - \beta\omega^2)(r - \gamma\omega^2) - \alpha\delta\omega^4\right) - i \omega^3 \left((\delta - \alpha\gamma)\omega^2 + (\alpha r - \beta\delta)\right)}{(r - \gamma\omega^2)^2 + \delta^2\omega^2}.\\
\end{align*}

In the case of assertion $(iii)$, since $\delta=\alpha\gamma$ and $\alpha r=\beta\delta$ then $\text{Im}(z^2(\omega ))=0$ for all $\omega\in\mathds{R}$. Moreover, taking $r=\beta\gamma$ and $\delta =\alpha \gamma$ we obtain: 
\begin{align*}
    & z^2(\omega) =\frac{\omega^4 - \beta\omega^2 - i\alpha\omega^3}{r - \gamma\omega^2 + i\delta\omega} =\frac{\omega^4-\beta \omega^2-i\alpha \omega^3}{\beta\gamma-\gamma \omega^2+i\alpha\gamma \omega }\\
     & =\frac{\omega^2(\omega^2-\beta -i\alpha \omega)}{\gamma(\beta-\omega^2+i\alpha \omega) }=\frac{-\omega^2}{\gamma}. 
\end{align*}
And therefore, for any $\lambda =\omega i \in i\mathds{R}$, if $z(\omega)$ is such that $G_h(\lambda, z(\omega ))=0$ then $\text{Re}(z(\omega))=0$. So that by Corollary \ref{corolario para polinomios hacia adelante} the forward semidiscretization cannot be Devaney chaotic. \\

For assertion $(ii)$ we may consider three cases.
\begin{itemize}
    \item \textit{Case 1:} $\delta \neq \alpha\gamma$ and $(\delta - \alpha\gamma) + (\alpha r - \beta\delta) \neq 0$. In this case, taking $\omega =1$ we have that $\text{Im}(z^2(1))\neq 0$. Hence for $\lambda=i\in i\mathds{R}$, there exists some $z_0\in\mathds{C}$, with $\text{Re}(z_0)<0$, such that $G_h(\lambda, z_0)=0$ and $\text{disc}_z(G_h)(i)\neq 0$, so that the conclusion holds by Corollary \ref{corolario para polinomios hacia adelante}.
    \item \textit{Case 2:} $\delta \neq \alpha\gamma$ and $(\delta - \alpha\gamma) + (\alpha r - \beta\delta) = 0$. Now we take $\omega =2$, so that  $4(\delta-\alpha\gamma)+(\alpha r-\beta\delta)\neq 0$ and $\text{Im}(z^2(2))\neq 0$. Consequently, for $\lambda=2i\in i\mathds{R}$ the conclusion follows as in the first case. 
     \item \textit{Case 3:} $\delta = \alpha\gamma$ and $(\alpha r - \beta\delta) \neq 0$. In this case we have that $\text{Im}(z^2(\omega))\neq 0$ for any $\omega\neq 0$, so that the conclusion holds as in the previous cases taking any $\lambda=\omega i \in i\mathds{R}$ with $\omega\neq 0$.  
\end{itemize}
\end{proof}
\subsection{Viscous van Wijngaarden-Eringen equation}
For $a_0>0$ and $Re_d>0$ we consider the following equation: 
\begin{equation}\label{wijngaarden}
\partial_t^2 u-\partial_x^2 u=(Re_d)^{-1}\partial_t\partial_x^2u+a_0^2\partial_t^2\partial_x^2u.
\end{equation}
Which we can write as the following system: 
\begin{equation}\label{sistema ecuaciones wijngaarden}
\begin{cases}
    &u_1=u\\
    &\partial_t u_1=u_2\\
    &\partial_t(I-a_0^2\partial_x^2)u_2=\partial_x^2u_1+(Re_d)^{-1}\partial_x^2
\end{cases}
\end{equation}
And therefore for $\rho>0$ such that $\rho<\frac{1}{a_0}$ we have the following abstract Cauchy problem: 
\begin{equation}\label{Cauchy Wijngaarden-Eringen}
     \frac{\partial}{\partial t}\begin{pmatrix}
         u_1\\
         u_2\\
     \end{pmatrix}=\begin{pmatrix}
0 & I &  \\
(I-a_0^2\partial_x^2)^{-1}\partial^2_{x} & (I-a_0^2\partial_x^2)^{-1}(Re_d)^{-1}\partial_x^2 
\end{pmatrix}\begin{pmatrix}
    u_1\\
    u_2\\
\end{pmatrix},\quad \begin{pmatrix}
    u_1(0)\\
    u_2(0)\\
\end{pmatrix}\in X_{\rho}^2.\\
 \end{equation}
 As in the Moore-Gibson-Thompson equation, the following theorem shows that no additional conditions are required to guarantee chaos in the projection family of operators associated to the solution semigroup. 
 \begin{theorem}\label{caos wijngaarden}
     For each $0<\rho<\frac{1}{a_0}$, the projection family associated to the solution semigroup of the viscous van Wijngaarden-Eringen equation is Devaney chaotic on $X_{\rho}$. 
 \end{theorem}
 \begin{proof} We will use Corollary \ref{corolario para racionales caso particular} identifying the matrix 
 $$A_d:=\begin{pmatrix}
0 & I &  \\
(I-a_0^2\partial_x^2)^{-1}\partial^2_{x} & (I-a_0^2\partial_x^2)^{-1}(Re_d)^{-1}\partial_x^2 
\end{pmatrix}$$
and the polynomials $p_1(z)=z^2$, $p_2(z)=(Re_d)^{-1}z^2$. As a consequence we have: 
$$P(\lambda, z)=z^2+\lambda(Re_d)^{-1}z^2-\lambda^2(1-a_0^2z^2).$$
Now for any $\lambda=bi$, $b\in\mathds{R}$ we have that 
$$P(bi,z)=0\iff z=\pm\sqrt{\frac{(bi)^2}{1+(Re_d)^{-1}bi+(bi)^2a_0^2}}=\frac{bi}{\pm\sqrt{1+(Re_d)^{-1}bi-a_0^2b^2}},$$
so that there is some small enough $b_0\in\mathds{R}$ such that 
$$\left|\frac{b_0i}{\sqrt{1+(Re_d)^{-1}b_0i-2a_0^2b_0^2}}\right|<\rho.$$
Now taking $\lambda_0:=b_0i$ and $z_0:=\frac{b_0i}{\sqrt{1+(Re_d)^{-1}b_0i-a_0^2b_0^2}}$ the conclusion follows by Corollary \ref{corolario para racionales caso particular}.
 \end{proof}
 Now we may take the semidiscretization of the viscous van Wijngaarden-Eringen equation. Let us note that for each $h>2a_0$:
 $$\{z\in\mathds{C}\text{ such that }1-a_0^2z^2=0\}\subset \mathds{C}\setminus\left(\overline{B(1/h, 1/h)}\cup\overline{B(-1/h, 1/h)}\right).$$
 And therefore we can express the semidiscretization for a spatial step $h>2a_0$ as: 
 \begin{equation}\label{Cauchy Wijngaarden-Eringen semidiscretization}
     \frac{\partial}{\partial t}\begin{pmatrix}
         u_1\\
         u_2\\
     \end{pmatrix}=\begin{pmatrix}
0 & I &  \\
(I-a_0^2\Delta_h^2)^{-1}\Delta_h & (I-a_0^2\Delta_h^2)^{-1}(Re_d)^{-1}\Delta_h^2 
\end{pmatrix}\begin{pmatrix}
    u_1\\
    u_2\\
\end{pmatrix},\quad \begin{pmatrix}
    u_1(0)\\
    u_2(0)\\
\end{pmatrix}\in X^2,\\
 \end{equation}
 where $X=\ell^p(\mathds{N}_0)$ or $c_0(\mathds{N}_0)$. 
 \begin{theorem}
     The following assertions hold: 
     \begin{enumerate}[(i)]
         \item For each $h>2a_0$, the backward semidiscretization of the viscous van Wijngaarden-Eringen equation is not chaotic on $\ell^p(\mathds{N}_0)$ $1\leq p<\infty$ or $c_0(\mathds{N}_0)$. 
         \item Given $a_0, Re_d>0$ and $h>0$ such that $2a_0<h<1/Re_d$ then the projection family of the forward discretization of the viscous van Wijngaarden-Eringen equation is Devaney chaotic on $\ell^p(\mathds{N}_0)$ or $c_0(\mathds{N}_0)$.  
     \end{enumerate}
 \end{theorem}
 \begin{proof}
     The first assertion is a direct consequence of Theorem \ref{backward semidiscretization}. \\
For the second assertion, we will use Corollary \ref{corolario para polinomios hacia adelante} identifying $p_1(z)=z^2$, $p_2(z)=(Re_d)^{-1}z^2$ and $q(z)=1-a_0^2z^2$. Hence we have: 
$$P_h(\lambda, z)=z^2(1+(Re_d)^{-1}\lambda+\lambda^2a_0^2)-\lambda^2.$$
As in the proof of Theorem \ref{caos wijngaarden}, for each $\lambda=bi$ with $b\in\mathds{R}_{+}$ we have: 
\begin{equation}\label{ec-zb}
    P_h(bi, z_b)=0\iff z_b=\frac{\pm bi}{\sqrt{1+(Re_d)^{-1}bi-a_0^2b^2}}=\frac{\pm bi}{w_b},
\end{equation}
where $w_b:=\sqrt{1+(Re_d)^{-1}bi-a_0^2b^2}$. Now we have that $z_b\in B(-1/h, 1/h)$ if and only if $|z_b+1/h|<\frac{1}{h}$, or equivalently $|z_b+1/h|^2<\frac{1}{h^2}$. Developing the last inequation we obtain: 
$$\left|z_b+\frac{1}{h}\right|^2=(\text{Re}(z_b)+1/h)^2+\text{Im}(z_b)^2=\text{Re}(z_b)^2+2\frac{\text{Re}(z_b)}{h}+\frac{1}{h^2}+\text{Im}(z_b)^2<\frac{1}{h^2},$$
which is equivalent to the following expression: 
\begin{equation}\label{ec-f1}
|z_b|^2+\frac{2\text{Re}(z_b)}{h}<0 \iff h<\frac{2\text{Re}(z_b)}{-|z_b|^2}.\end{equation}
Denote $w_b:=u_b+v_b i$ and take the positive branch of the square root in \eqref{ec-zb}. Then we have 
$$z_b=\frac{bi}{w_b}=\frac{bi(\overline{w_b})}{|w_b|^2}, $$
and therefore
$$|z_b|^2=\frac{b^2}{|w_b|^2}, \quad\text{ Re}(z_b)=\frac{- bv_b}{|w_b|^2}.$$
So that we can express condition \eqref{ec-f1} as: 
\begin{equation}\label{ec-f1.1}
    h<\frac{ 2v_b}{b}.
\end{equation}
Now recall that for $s+ti\in\mathds{C}$, we have that $\sqrt{s+ti}=u+vi$, where $u, v\in\mathds{R}$ are given as: 
$$u=\sqrt{\frac{s+\sqrt{s^2+t^2}}{2}},\quad v=\frac{t}{|t|}\sqrt{\frac{-s+\sqrt{s^2+t^2}}{2}}.$$
As a consequence, taking $s_b:=1-a_0b^2$ and $t_b:=\frac{1}{Re_d}b$ we obtain:
\begin{equation}\label{ec-f2}
    u_b=\sqrt{\frac{s_b+\sqrt{s_b^2+t_b^2}}{2}}=\sqrt{\frac{t_b^2}{2\left(\sqrt{s_b^2+t_b^2}-s_b\right)}}=t_b\sqrt{\frac{1}{2\left(\sqrt{s_b^2+t_b^2}-s_b\right)}},
\end{equation}
\begin{equation}\label{ec-f3}
    v_b=\frac{t_b}{|t_b|}\sqrt{\frac{-s_b+\sqrt{s_b^2+t_b^2}}{2}}=
    t_b \sqrt{\frac{1}{2\left(\sqrt{s_b^2+t_b^2}+s_b\right)}}.
\end{equation}
This lead us to conclude that for $b>0$ such that $z_b=\frac{b_i}{w_b}$ and $h>0$ with 
$$h<\frac{2v_b}{b}=\frac{2}{Re_d}\sqrt{\frac{1}{2\left(\sqrt{s_b^2+t_b^2}+s_b\right)}}=\frac{1}{Re_d}\sqrt{\frac{2}{\left(\sqrt{(1-a_0^2b^2)^2+(b/Re_d)^2}+(1-a_0^2b^2)\right)}},$$
then $P_h(bi, z_b)=0$ and $z_b\in B(-1/h, 1/h)$. 
Finally, we observe that: 
$$\lim_{b\rightarrow 0}\frac{1}{Re_d}\sqrt{\frac{2}{\left(\sqrt{(1-a_0^2b^2)^2+(b/Re_d)^2}+(1-a_0^2b^2)\right)}}=\frac{1}{Re_d}.$$
By the hypothesis $h<\frac{1}{Re_{d}}$, and therefore there exists some $b_0>0$ such that 
$$h<\frac{2v_{b_0}}{b_0}
=\frac{1}{Re_d}\sqrt{\frac{2}{\left(\sqrt{(1-a_0^2b_0^2)^2+(b_0/Re_d)^2}+(1-a_0^2b_0^2)\right)}}$$
and the conclusion follows. 
 \end{proof}
	

\begin{thebibliography}{99}
			
			\bibitem{BM2}
			J. Banasiak and M. Moszyński
			\newblock {A generalization of Desch–Schappacher–Webb criteria for chaos}. 
			\newblock {\em Discrete and Continuous Dynamical Systems}, 12 (2005), 959--972.
			
			\bibitem{Banks-Brook}
			J. Banks, J. Brooks, G. Cairns, G. Davis, and P. Stacey.
			\newblock On Devaney's definition of chaos.
			\newblock {\em Amer. Math. Monthly}, 99(4) (1992), 332--334.

\bibitem{bayart-bermudez}
F. Bayart and T. Bermúdez. 
\newblock{Semigroups of chaotic operators.} 
\newblock{\em Bull. Lond. Math. Soc.},  41(5) (2009), 823--830.

\bibitem{bayartgrivauxhyperci} F. Bayart and S. Grivaux. 	
   \newblock{Hypercyclicit\'e: le r\^ole du spectre ponctuel unimodulaire.}
			\newblock{\em C. R. Math. Acad. Sci. Paris}, 338 (2004).

    
			\bibitem{bayart}
            F. Bayart and E. Matheron. \newblock{Dynamics of Linear Operators.} 
            \newblock{Cambridge University Press, 2009.} 
   
  \bibitem{brezis2011functional}
			H. Brezis. 
			\newblock{Functional Analysis, Sobolev Spaces and Partial Differential
				Equations.}
			\newblock{Springer, New York, 2011.}
			
 
			\bibitem{conejerol_lizama_murillo2016vanwjingaarne} J. A. Conejero, C. Lizama and M. Murillo-Arcila.
		On the existence of chaos for the Viscous Van Wjingaarden Equation. {\it Chaos, Solitons and Fractals,}  89 (2016).
			
			\bibitem{Conejero_lizama_murillo2017Chaotic} J. A. Conejero, C. Lizama, and M. Murillo-Arcila. Chaotic semigroups from second order
			partial differential equations. {\it J. Math. Anal. Appl.,} 456 (2017).
			
			
			\bibitem{conejero_lizama_rodenas2015chaotic}
			J.~A. Conejero, C.~Lizama, and F.~R\'odenas.
			\newblock Chaotic behaviour of the solutions of the {M}oore-{G}ibson-{T}hompson
			equation. 
			\newblock {\em Appl. Math. Inf. Sci.}, 9 (5) (2015).
			


			\bibitem{conejero_martinez-gimenez_peris_rodenas2016chaotic}
			J.~A. Conejero, F.~Mart{\'{\i}}nez-Gim{\'e}nez, A.~Peris, and F.~R{\'o}denas.
			\newblock Chaotic asymptotic behaviour of the solutions of the
			{L}ighthill-{W}hitham-{R}ichards equation.
			\newblock {\em Nonlinear Dynam.}, 84 (1) (2016). 
			
            
	
			\bibitem{conejero_peris_trujillo2010chaotic}
			J.~A. Conejero, A.~Peris, and M.~Trujillo.
			\newblock Chaotic asymptotic behavior of the hyperbolic heat transfer equation
			solutions.
			\newblock {\em Internat. J. Bifur. Chaos}, 20(9) (2010).
			
		
			
            \bibitem{cooke}
            K. L. Cooke.
\newblock{Differential-Difference Equations.}
International Symposium on Nonlinear Differential Equations and Nonlinear Mechanics,
Academic Press,
1963.
            
            
            \bibitem{conway-functional} 
			J. B. Conway. 
			\newblock{A course in Functional Analysis.}
			\newblock{Graduate Texts in Mathematics, Springer, 1990}. 
			
			
			
			\bibitem{desch_schappacher_webb1997hypercyclic}
			W.Desch, W.Schappacher and G.F. Webb.
			\newblock{Hypercyclic and chaotic semigroups of linear operators.}
			\newblock{\em Ergodic Theory Dynam. Systems}, 17(4) (1997).
			
			\bibitem{Devaney}
			R. Devaney.
			\newblock {An Introduction to Chaotic Dynamical Systems}.
			\newblock{CRC press, 2018.}
		
			\bibitem{dunford-schwartz}
			N. Dunford and J.T. Schwartz.
			\newblock{Linear Operators},
			\newblock{Interscience publishers}, New York, 1957. 
			
			\bibitem{engel-nagel}
			K. J. Engel and R.Nagel.
			\newblock {One-parameter Semigroups for Linear Evolution Equations.} {Graduate Texts in Mathematics}.
			\newblock{Springer-Verlag, New York, 2000.}


			
			
			\bibitem{Alfred}
			K. G. Grosse-Erdmann and A. Peris.
			\newblock {Linear Chaos}.
			\newblock {Universitext. Springer, London, 2011.}


			
			
			
			
			
			\bibitem{Herzog}
			G. Herzog.
			\newblock{On a universality of the heat equation.}
			\newblock {\em Math. Nachr.}, 188 (1997).
			
			\bibitem{kaup}
			B. Kaup and L. kaup. 
            \newblock{Holomorphic Functions of Several Variables}. 
            \newblock{Walter de Gruyter, Berlin-New York, 1983.}


			
		\bibitem{lizamamurillo2023} 
			C. Lizama and M. Murillo-Arcila.
			\newblock On the dynamics of the Damped Extensible Beam 1D-equation.
			\newblock {\em J. Math. Anal. Appl.}, 522 (2023).

            \bibitem{lizamamurilloMGT}
            C. Lizama and M. Murillo-Arcila. 
            \newblock On the existence of chaos for the fourth-order Moore-Gibson-Thompson equation. 
            \newblock{\em Chaos Solitons Fractals}, 176 (2023).

            \bibitem{lizamamurillo2024}
            C. Lizama and Marina Murillo-Arcila. 
            \newblock{On semidiscrete models dominated by the heat, wave and Laplace equations.}
            \newblock{\em Discrete and Continuous Dynamical Systems,} 44(8) (2024).

            \bibitem{LMP}
            C. Lizama, M. Murillo-Arcila and J. Puerta-Fern\'andez. 
            \newblock{From stability to chaos: a complete classification of the damped Klein-Gordon dynamics.}
            \newblock{\em Math. Methods Appl. Sci.}, 49(9) (2026).

            \bibitem{LMV}
            C. Lizama, M. Murillo-Arcila and \'A. Vargas-Moreno.
            \newblock On the dynamics of a class of partial differential equations.
            \newblock {\em J. Math. Anal. Appl.}, 545 (2025).
			
	

            \bibitem{MPV-chaos}
            M. Murillo-Arcila, A. Peris and Á. Vargas-Moreno.
            \newblock{Chaotic finite difference operators}, \newblock{\em Chaos},  33(9) (2023).
			
			
		
			
			\bibitem{Pa}
			A. Pazy.
			\newblock{Semigroups of Linear Operators and Applications to Partial Differential Equations.} Springer, New York, 1983.
            
\bibitem{MOL}
W. E. Schiesser and G. W. Griffiths.  \newblock{A Compendium of Partial Differential Equation Models: Method of Lines Analysis with Matlab,} Cambridge University Press, 2009.

\bibitem{slavik}
A. Slavík and P. Stehlík. 
\newblock{Dynamic diffusion-type equations on discrete-space domains.}
\newblock{\em Journal of Mathematical Analysis and Applications},
427(1) (2015).


\bibitem{zhu}
P.X. Zhu and Q.M. Xiang.
\newblock{Devaney chaos in high-dimensional linear fourth-order PDE.}
\newblock{\em Qual. Theory Dyn. Syst.}, 25(1) (2026).


            \bibitem{yang}
            P.X. Zhu, Q. Yang and C.L. Zhang. 
            \newblock{The mechanism of chaotic complexity in high-dimensional linear third-order partial differential equation.}
            \newblock{\em Math. Methods Appl. Sci.},   49(5) (2026).                 
		\end{thebibliography}
\end{document}